\documentclass[reqno,10pt]{amsart}
\usepackage{amssymb}
\usepackage{amsmath,amsfonts,amsthm,amssymb}
\usepackage{indentfirst}
\usepackage{paralist}
\usepackage{graphics} 
\usepackage{epsfig} 
\usepackage[colorlinks=true]{hyperref}
\usepackage[english]{babel}
\usepackage{amsmath,amssymb}

\hypersetup{urlcolor=red, citecolor=blue}
\newtheoremstyle{theorem}
  {10pt}          
  {10pt}  
  {\sl}  
 {}
  {\bf}  
  {. }    
  { }    
  {}     
\theoremstyle{theorem}

\newtheorem{theorem}{Theorem}[section]
\newtheorem{corollary}{Corollary}[section]

 \newtheorem{lemma}{Lemma}[section]
 
 \newtheorem{remark}{Remark}[section]
  
\numberwithin{equation}{section}

\newtheoremstyle{defi}
  {10pt}          
  {10pt}  
  {\rm}  
  {}  
  {\bf}  
  {. }    
  { }    
  {}     
\theoremstyle{defi}

\newcommand{\mr}{\mathbb{R}}
\newcommand{\vp}{\varphi}
\newcommand{\me}{\mathbb{E}}
\newcommand{\m}{\mathbb}
\newcommand{\p}{\partial}
\newcommand{\ee}{\eta^\epsilon}
\newcommand{\eps}{\epsilon}

\allowdisplaybreaks

\begin{document}
\baselineskip = 20 pt
\title[stochastic Camassa-Holm equation]
{Modulation analysis of the stochastic Camassa-Holm equation with pure jump noise}%
\author[Y. Chen]{Yong Chen$^{1}$}%
\author[J. Duan]{Jinqiao Duan$^{2}$}
\author[H. Gao]{Hongjun Gao$^{3}$}
\author[X. Guo]{Xingyu Guo$^{1}$}
\address{1. Department of Mathematics, Zhejiang Sci-Tech University, Hangzhou 310018, PR China}

\address{2. Department of Applied Mathematics and Department of Physics, Illinois Institute of Technology, Chicago, IL 60616, USA}
\address{3. School of Mathematics, Southeast University, Nanjing 211189, PR China}


\thanks{E-mail: youngchen329@126.com,\ duan@iit.edu,\ gaohj@hotmail.com}


\begin{abstract}
We study the stochastic Camassa-Holm equation with pure jump noise.
We prove that if the initial condition of the solution is a solitary wave solution of the unperturbed equation,
the solution decomposes into the sum of a randomly modulated solitary wave and
 a small remainder. Moreover, we derive the equations for the modulation parameters
 and show that the remainder converges to the solution of a stochastic linear equation
 as amplitude of the jump noise tends to zero.

\end{abstract}
\maketitle
\section{Introduction}

\subsection{Background}

The Camassa-Holm (CH) equation
\begin{equation}\label{11a}
u_t-u_{xxt}+3uu_x+2k u_x-2u_xu_{xx}-uu_{xxx}=0,\; t>0,\;x\in\mr,
\end{equation}
with $k\geq0,$ was derived by Camassa and Holm in
\cite{camassa} as a model of shallow water waves.
Here $u$ denotes the fluid velocity in the $x$ direction or,
equivalently, the height of the water's free surface above a flat
bottom \cite{camassa,Johnson02}. Eq. \eqref{11a} was originally derived by Fuchssteiner and Fokas
\cite{Fuchssteiner,Fuchssteiner96} as a bi-Hamiltonian generalization
of KdV. A rigorous justification of the derivation of Eq. \eqref{11a} as an approach to the governing equations for water waves was recently provided by Constantin and Lannes \cite{cons2009}.
Eq. \eqref{11a} was also arisen as an equation of the geodesic
flow for the $H^1$ right-invariant metric on the Bott-Virasoro group (if $k > 0$) \cite{Misi} and on the
diffeomorphism group (if $k = 0$) \cite{CoKo,CoKo1}.

The CH equation \eqref{11a} is completely integrable \cite{camassa,Const01}, which has the bi-Hamiltonian
structure \cite{camassa,Fuchssteiner}
\begin{align}
&m_t=\mathcal{J}_1\frac{\delta H_1[m]}{\delta m}=\mathcal{J}_2\frac{\delta H_2[m]}{\delta m},\label{i6}\\
&\mathcal{J}_1=-(\p_xm+m\p_x+2k\p_x),\;\; \mathcal{J}_2=-(\p_x-\p_x^3)
\end{align}
with a momentum density
$m:=u-u_{xx}$ and the two Hamiltonians
\begin{align}
&H_1[m]=H_1(u)=\frac12\int_\mr u^2(x)+u_x^2(x)dx=\frac12\int_\mr umdx,\label{i1}\\
&H_2[m]=H_2(u)=\frac12\int_\mr u^3(x)+u(x)u_x^2(x)+2k u^2(x)dx.\label{i1a}
\end{align}
The CH equation \eqref{11a}  can be written in Hamiltonian form as
\begin{align}\label{i2}
\partial_t\frac{\delta H_1(u)}{\delta u}=-\partial_x\frac{\delta H_2(u)}{\delta u}.
\end{align}

The Cauchy problem for the CH equation \eqref{11a} has been studied extensively. For initial data $u_0\in H^s(\mathbb{R}), s>3/2$,
Eq. \eqref{11a} is locally well posed \cite{Cons98475,danchin01,lio00}.
Moreover, Eq. \eqref{11a} has global strong solutions \cite{Cons98475,Cons00321} and also finite time blow-up solutions \cite{Cons98475,Cons98229,Cons00321,Cons0075,danchin01,lio00}. On the other hand, it
has global weak solutions in $H^1(\mathbb{R})$ \cite{Bressan07,Bressan071,Coclite06,cones98,Const0045,Holden0739,Holden07,Holden09,xin00}.
The ill-posedness of the CH equation in $H^{3/2}$ and in the critical space $B^{3/2}_{2,r}, 1 < r <\infty$ is proved in \cite{GLMY19}.

The CH equation \eqref{11a} possesses smooth solitary-wave solutions called
solitons if $k>0$ \cite{camassa94} or peaked solitons if $k=0$ \cite{camassa}.
In particular, when $k > 0$,
Eq. \eqref{11a} possesses the smooth soliton with the expression $u(t,x)=\varphi_c(x-ct+x_0), c>2k, x_0\in\mr$,
in a parametric form as follows \cite{Jo,LiZh}
\begin{align*}
&u(t,x)=\frac{c-2k}{1+(2k/c)\sinh^2\theta},
\end{align*}
where
\begin{align*}
\theta=\frac1{2\sqrt{k}}\sqrt{1-\frac{2k}c}(y-c\sqrt{k}t),\;
x=\frac{y}{\sqrt{k}}+\ln\frac{\cosh(\theta-\theta_0)}{\cosh(\theta+\theta_0)},\;
\theta_0=\tanh^{-1}\sqrt{1-\frac{2k}c}.
\end{align*}
The  soliton $\varphi_c$ satisfies the following equation
\[
-c\varphi_c+c\varphi_{cxx}+\frac32\varphi_c^2+2k\varphi_c=\varphi\varphi_{cxx}+\frac12\varphi_{cx}^2, x\in\mr.
\]
It is shown in \cite{camassa94,constr02} that the CH equation \eqref{11a} admits the smooth and positive solitary
wave solution $\vp_c$ with an even profile decreasing from its peak height $c-2k.$  Moreover, $|\vp_c'|\leq \vp_c$ and as $|x| \to \infty$,
\begin{align}
\vp_c(x)=O(\exp(-\sqrt{1-\frac{2k}c}|x|)).
\end{align}
The CH equation \eqref{11a} has even more complex solutions, such as multi-solitons
which can be given also in a parametric form like one soliton \cite{CaHoTi,CoGeIv,LiZh,Mats}.

Some kinds of the stability of solitons of the CH equation are considered in \cite{constr02,dika07,WL22}.
In \cite{constr02}, the nonlinear stability of the 1-solitons is proved by applying the general spectral method.
In \cite{dika07}, the orbital stability of the train of N solitary waves of the CH equation in energy space $H^1$ is investigated .
In \cite{WL22}, the dynamical stability of smooth N-solitons of the CH equation in Sobolev space $H^N$ and the orbital stability of smooth double solitons
in the space $H^2$ are proved for all the time.
The stability of peakons and multipeakons of the CH equation are established in \cite{constr00,dm09}.

\subsection{Stochastic Camassa-Holm equation}

Since there are some uncertainties in geophysical and
climate dynamics \cite{ad01,ddh15,duan}, it is widely recognized to
take random effect into account in mathematical models.
Using stochastic variational method \cite{ddh15,ddh16}, the following stochastic CH equation with Brownian motion was derived in \cite{cddh18},
\begin{equation}\label{sch-bm}
dm+(\p_xm+m\p_x)v=0,
\end{equation}
where $m=u-u_{xx}$ and the stochastic vector field $v$, defined by
\[
v(x,t)=u(x,t)dt+\sum_{i=1}^N\xi_i(x)\circ dW^i_t.
\]
The vector $v(t, x)$ represents random spatially correlated shifts in
the velocity, the functions $\xi_i(x), i =
1, 2, \ldots, N$ are the spatial correlations, $W^i_t
, i = 1, 2, \ldots, N$ are independent Brownian motions and $\circ$ is the Stratonovich product. Eq. \eqref{sch-bm} admits peakon solutions and
isospectrality \cite{cddh18}.
Applying a finite element discretization,
 the authors in \cite{BCH} reveal that peakons can still form in the presence of stochastic perturbations.
 The local existence and uniqueness of strong solution of \eqref{sch-bm} is proved in \cite{AlBrDa}.
The well-posedness of stochastic CH equation with the initial value $u_0\in H^s, s>3/2$ is proved in \cite{Chen12,tang18}.
The martingale solution in $H^1$
is established in \cite{CDG21a} under the condition that $m_0=u_0-u_{0xx}$ is a positive
 regular Borel measures, which is improved in \cite{GaKaPa22} with $u_0\in H^1$.
The global well-posedness of the viscous Camassa-Holm equation with gradient noise
is established in \cite{HoKaPa22}.

The analysis of the L\'evy noise is different from \cite{ddh15,ddh16}
and is motivated by the requirement that the noise must preserve some invariance properties of stochastic Hamiltonian structure.
Thus, it needs to find an analogue of the Stratonovich integral with respect
to compensated Poisson random measure.  The work of Marcus \cite{mau81}, developed later by Applebaum \cite{app09},
Kunita \cite{{kur04}}, Chechkin and Pavlyukevich \cite{Chechkin}, provides a framework to resolve this technical issue.
Using stochastic variational method \cite{ddh15,ddh16}, we derive the following stochastic CH equation with pure jump noise \cite{chendg},
\begin{equation}\label{sch-le}
dm+(um_x+2mu_x)dt+ m\diamond dL(t)=0,
\end{equation}
where $ "\diamond"$ is the Marcus product 
and $L(t)$ is a L\'{e}vy motion  with pure jump.
The local well-posedness and large deviations of \eqref{sch-le} with the initial value $u_0\in H^s, s>3/2$ are given
in \cite{chgao17} as an example. The global existence, the wave breaking phenomena and moderate deviations  of \eqref{sch-le}
with $u_0\in H^s, s>3/2$
are proved in \cite{chendg}.

Consider the following stochastic Hamiltonian
\begin{align*}
&\tilde{H}_1[m]dt=\frac12\int_\mr u(t,x)m(t,x)dxdt+ \int_\mr m(t,x)\sigma(x)\diamond dL(t),
\end{align*}
where $\sigma=\sigma(x)$ depends on $x\in\mr$,  $"\diamond"$ is the Marcus product 
and $L(t)$ is a L\'{e}vy motion  with pure jump, defined as
\begin{equation}\label{il1}
L(t)=\int_0^t\int_\mathbb{Z}z\tilde{\mathcal{N}}(dt,dz).
\end{equation}
Here $\mathbb{Z}=\{z: |z|\leq1\}$ is a locally compact Polish space, and $\tilde{\mathcal{N}}$ is the
compensated Poisson random measure.
Let $\mathcal{N}$ be a Poisson random measure on $[0,T]\times \mathbb{Z}$ with
a $\sigma$-finite intensity measure $\lambda_T\otimes\vartheta$,
where $ \lambda_T$ is the Lebesgue measure
on $[0,T]$ and $\vartheta$ is a $\sigma$-finite measure on $\mathbb{Z}$, such that
\begin{equation}\label{i5}
\int_{\m{Z}}e^{2\|\sigma_x\|_{L^\infty}|z|}\vartheta(dz)<\infty,\;\;
\int_{\m{Z}}z^2\vartheta(dz)<\infty.
\end{equation}
Then, for $A\in \mathcal{B}(\mathbb{Z})$ with $\vartheta(A)<\infty$,
\begin{equation*}
\tilde{\mathcal{N}}([0,t]\times A)=\mathcal{N}([0,t]\times A)-t\vartheta(A).
\end{equation*}
We derive the stochastic CH equation
\begin{align}\label{d2}
dm=&\mathcal{J}_1\frac{\delta \tilde{H}_1[m]}{\delta m}dt=\mathcal{J}_1(udt+\sigma \diamond dL(t))\nonumber\\
=&-(um_x+2mu_x+2ku_x)dt-(\sigma m_x+2m\sigma_x+2k\sigma_x)\diamond dL(t).
\end{align}
By the relation between the advective form $u$ and the moment $m$, applying $(1-\partial_x^2)^{-1}$ to both sides of \eqref{d2} to yield
\begin{align}
&du+(uu_x+P_x)dt+\zeta(u)\diamond dL(t)=0,\label{d6ge}\\
&(1-\p_x^2)P=u^{2}+\frac12u^{2}_x+2ku,\label{p}
\end{align}
where
\begin{align}\label{zeta}
\zeta(u)=3\sigma u_x-2\sigma_xu+(1-\p_x^2)^{-1}(2\sigma_xu-u\p_x^3\sigma +\sigma_{xx}u_x)+2k(1-\p_x^2)^{-1}\sigma_x.
\end{align}
The elliptic equation for $P$ can be solved to supply
\[
P=K*(u^{2}+\frac12u^{2}_x+2ku),\;K(x)=\frac12e^{-|x|}.
\]
Mathematically, as explained in Remark \ref{remark1}, the nonlocal part and the part $2u\sigma_x$ of the noise term in \eqref{zeta}
offers no new (essential) difficulties compared to $3\sigma u_x$. For the sake of clarity,
we will therefore focus on the following stochastic CH equation with gradient pure jump noise
\begin{align}
&du+(uu_x+P_x)dt+\sigma u_x\diamond dL(t)=0,\label{d6a}
\end{align}
where $P$ is given by \eqref{p}.
For fixed $z\in\m{Z}, w\in L^2$, let $\Phi(t,z, w)$ be
the solution of the following partial differential equation
\[
dy(t)=-z\sigma y_x(t)dt, \;\;y(0)=w.
\]
Then $\Phi(t,x,w)=w(x-\sigma zt)$.
Hence, the stochastic integration in equation \eqref{d6a} with Marcus form
can be written with $\Phi(1, z, u)$ as follows \cite{app09}
\begin{align}\label{d6}
du+(uu_x+P_x)dt=&\int_\m{Z}[u(t-, x-\sigma z)-u(t-,x)] \tilde{\mathcal{N}}(dt,dz)\nonumber\\&
+\int_\m{Z}[u(t, x-\sigma z)-u(t,x)+z\sigma u_x]\vartheta(dz)dt.
\end{align}

The literature
on stochastic partial differential equations driven by L\'{e}vy noise in the "Marcus" canonical
form is very limited and such work has recently been initiated by Brze\'{z}niak et al.
in \cite{brzeslaw19,brezniak19a}. The noise intensity $\sigma$ is a constant in \cite{ChLi}.
The operator in the Murcus integral in \cite{brzeslaw19,brezniak19a} is bounded, while it is
unbounded here. Moreover, in this form of noise, it can maintain the invariants $H_1(u)$ and $H_2(u)$.


\subsection{Our aims}

Our aim is to investigate the influence of random perturbations with the form given in Eq. \eqref{d6} to the smooth solitary wave of the CH equation.
More precisely, we study the stochastic CH equation
\begin{align}\label{d2-ch}
&du+(uu_x+P_x)dt+\epsilon\sigma u_x\diamond dL(t)=0,
\end{align}
where $P$ is given by \eqref{p} and $\epsilon>0$  is a small parameter.
The stochastic integration in equation \eqref{d6a} with Marcus form
can be written as follows \cite{app09}
\begin{align}\label{d2-cha}
du+(uu_x+P_x)dt=&\int_\m{Z}[u(t-, x-\eps\sigma z)-u(t-,x)] \tilde{\mathcal{N}}(dt,dz)\nonumber\\&
+\int_\m{Z}[u(t, x-\eps\sigma z)-u(t,x)+\epsilon z\sigma u_x]\vartheta(dz)dt.
\end{align}
We will use the collective coordinate approach to investigate the
influence of random perturbations on the propagation of deterministic standing waves (e.g. \cite{Bouard07,Bouard09,wad84}).
This approach consists in writing that the main part of the solution is given by a modulated soliton and in finding then the modulation equations
for the soliton parameters. The modulation theory, in general, provides an
approximate and constructive answer to questions on concerning the location of the standing wave
and the behavior of its phase for $t> 0.$
The random modulations of solitons of the stochastic Korteweg-de Vries equation and stochastic Schr\"{o}dinger equation under the influence of the
Brownian motion have been studied in \cite{Bouard07,Bouard09,wad84}.
As far as we know, it's the first paper to consider the influence of the pure jump noise to the solitions.

Let $\vp_{c_0}(x)$ with $c_0>0$ fixed be a smooth solitary wave solution of equation \eqref{d2-ch} with $\epsilon=0$.
Define the functional
\begin{align}\label{lya}
H_c(u)=cH_1(u)-H_2(u), u\in H^1,
\end{align}
where $H_1$ and $H_2$ are given in \eqref{i1} and \eqref{i1a} respectively.
Note $\vp_{c_0}$ is a critical point of $H_{c_0}$, that is
\begin{align}\label{i4}
H_{c_0}'(\vp_{c_0})={c_0}H_1'(\vp_{c_0})-H_2'(\vp_{c_0})=0,
\end{align}
where $H_1'$ and $H_2'$ are the Fr\'echet derivatives of $H_1$ and $H_2$ in $H^1(\mr)$ respectively. The linearized Hamiltonian operator $\mathcal{L}_c$
of $cH_1'+H_2'$ around $\vp$ is defined by
\begin{align}\label{i5}
\mathcal{L}_c=-\partial_x((2c-2\vp_c)\p_x)-6\vp_c+2\p_x^2\vp_c+2(c-2k).
\end{align}
Denote $u^\epsilon(t,x)$ be the solution of equation \eqref{d2-ch} with the initial value $u^\epsilon(0,x) = \vp_{c_0}(x)$.
 We prove that $u^\epsilon(t,x)$ can be decomposed into the sum of a randomly modulated solitary wave $\vp_{c^\eps(t)}(x^{\eps}(t))$ and
 a small remainder $\eps\ee(t)$, which is randomly modulated in its phase $x^\epsilon(t)$ and frequency $c^\epsilon(t)$.
Then, we study more precisely the behavior at order one in $\epsilon$ of the remaining term $\ee$ in the
preceding decomposition as $\epsilon$ goes to zero.

%
%

\subsection{The main results}

Now, we give the main results of the paper.
The first one is that the main part of the solution of equation \eqref{d2-ch} is a solitary wave, randomly modulated in its phase $x^\epsilon(t)$ and frequency $c^\epsilon(t)$ as follows.

\begin{theorem} [Stochastic modulated solitary wave]\label{tm2}

For $\epsilon>0,$ let $u^\epsilon(t,x)$ be the solution of equation \eqref{d2-ch} with
$u_0(x)=\vp_{c_0}(x)$ and $c_0>k$.
Then, there exists $\alpha_0>0$ such that for $0<\alpha\leq\alpha_0$, there is a stopping time $\tau_\alpha^\epsilon$ a.s.
and there are semi-martingale processes  $c^\epsilon(t)$ and $x^\epsilon(t)$ defined a.s. for $t\leq\tau_\alpha^\epsilon,$
with values respectively in $\mr^{+}$ and $\mr$. The solution $u^\epsilon(t,x)$ can be decomposed as
\[
u^\epsilon(t,x)=\varphi_{c^\epsilon(t)}(x-x^\epsilon(t))+\epsilon\ee(t,x-x^\epsilon(t)).
\]
Let $\eta^\epsilon(t,x)=\frac1\epsilon[u^\epsilon(t,x+x^\epsilon(t))
-\varphi_{c^\epsilon(t)}(x)]$. Then it satisfies the orthogonality conditions
\begin{align}\label{orth1}
(\eta^\epsilon, (1-\p_x^2)\varphi_{c_0})=0,\;\;\hbox{and}\; \;
(\eta^\epsilon, (1-\p_x^2)\p_x\varphi_{c_0})=0.
\end{align}
Moreover, for $t\leq \tau_\alpha^\epsilon,$
\[
\|\epsilon\eta^\epsilon(t)\|_{H^1}\leq \alpha,\;\;\hbox{and}\;\;|c^\epsilon(t)-c_0|\leq \alpha,\;\;a.s.
\]

In addition, for any $T>0$ and $\alpha\leq \alpha_0$,  there is a $\epsilon_0>0$, for $\epsilon<\epsilon_0,$
\begin{align}\label{exit}
\m{P}(\tau_\alpha^\epsilon\leq T)\leq C\frac{T}{\alpha^4}b(\epsilon),
\end{align}
where $b(\epsilon)=\int_{\m{Z}}((e^{\eps|z|\|\sigma_x\|_{L^\infty}}-1)^2+(e^{\frac32\eps|z|\|\sigma_x\|_{L^\infty}}-1)^2)\vartheta(dz).$
\end{theorem}

The second one is the convergence of $\ee$ as $\epsilon$ goes to zero.

\begin{theorem}\label{th3}
Let $\ee, c^\epsilon$ and $x^\epsilon$ be given by Theorem \ref{tm2}, with $\alpha\leq \alpha_0$ fixed. Then, for any $T>0$, the process $\ee$
converges in probability in the space $\m{D}([0,T\wedge\tau_\alpha^\epsilon]; L^2)$, as $\epsilon$ goes to 0, to a process $\eta$ satisfying the linear eqaution
\begin{align}\label{eta}
 d\eta =&\frac12(1-\p_x^2)^{-1}\p_x\mathcal{L}_{c_0}\eta dt +( y(t) \p_x\vp_{c_0}-a(t)\p_c\vp_{c_0})dt\nonumber\\&
+(\sigma(x)\p_x\vp_{c_0}+\p_x\vp_{c_0}  \mu(t)-\p_c\vp_{c_0}b(t))\diamond dL(t),
\end{align}
with $\eta(0)=0$, where
\begin{align}
    y (t)=& -\frac12(\p_x\vp_{c_0}, (1-\p_x^2)\p_x\vp_{c_0})^{-1}(\p_x\mathcal{L}_{c_0}\eta, \p_x\vp_{c_0}),\label{eta1} \\
    \mu (t)=& -(\p_x\vp_{c_0}, (1-\p_x^2)\p_x\vp_{c_0})^{-1}(\sigma\p_x\vp_{c_0}, (1-\p_x^2)\p_x\vp_{c_0}),\\
    a (t)=&-\frac12(\p_c\vp_{c_0}, (1-\p_x^2)\vp_{c_0})^{-1} (\p_x\mathcal{L}_{c_0}\eta, \vp_{c_0}),\\
    b(t)=& (\p_c\vp_{c_0}, (1-\p_x^2)\vp_{c_0})^{-1}(\sigma\p_x\vp_{c_0}, (1-\p_x^2)\vp_{c_0}).\label{eta1x}
\end{align}

Moreover, the modulation parameters may be written  as
\begin{align}
&dx^\epsilon(t)=c^\epsilon(t) dt+\epsilon  y^\epsilon(t) dt+\epsilon \mu^\epsilon(t)\diamond dL(t),\label{xe1}\\
&dc^\epsilon(t)=\epsilon  a^\epsilon(t) dt+\epsilon b^\epsilon(t)\diamond dL(t)\label{ce1},
\end{align}
for some adapted processes $ y^\epsilon, \mu^\epsilon, a^\epsilon, b^\epsilon$ with values in $\mr$ satisfying:
$( y^\epsilon, a^\epsilon, b^\epsilon, \mu^\epsilon)\to ( y, a, b, \mu)$  in probability in
$\m{D}([0,T])$ as $\epsilon\to0$.
\end{theorem}


This paper is organized as follows.
In Section \ref{s2}, we justify the existence of the modulation parameters and give an estimate on the time up to
which the modulation procedure is available. In Section \ref{s3},  we give the equations of the modulation parameters.
In Section \ref{s4}, we show the convergence of the remainder term as $\epsilon$ to zero.

\smallskip

\section{Modulation and estimates on the exit time}\label{s2}

In this section, we prove the existence of modulation parameters and the estimate on the exit
time.
%
%
First, we present the following It\^o formula.

\begin{lemma}[It\^o formula, \cite{brzeslaw19}, Theorem B.2]\label{ito2}
Assume that $U$ is a Hilbert space. Let $Y$ be a $U$-valued process given by
\[
Y(t)=Y_0+\int_0^ta(Y(s))ds+\int_0^t\int_{\m{Z}}f(Y(s-))\diamond dL(s),\;\;t\geq0,
\]
where $a, f: U\to U$ are  $\mathcal{F}_t$-adapted random mappings.
Let $V$ be a separable Hilbert space. Let $\phi: U\to V$ be a function of class $C^1$
such that the first derivative $\phi':U\to L(U; V)$ is $(p-1)$-H\"older continuous. Then for every $t > 0$, we have $\m{P}$-a.s.
\begin{align*}
\phi(Y(t))=&\phi(Y_0)+\int_0^t\phi'(Y(s))(a(Y(s)))ds+\int_0^t\int_{\m{Z}}
[\phi(\Phi(1,z,Y(s-)))-\phi(Y(s-))]\tilde{\mathcal{N}}(ds,dz)\\
&+\int_0^t\int_{\m{Z}}
[\phi(\Phi(1,z,Y(s)))-\phi(Y(s))-z\phi'(Y(s))f(Y(s))]\vartheta(dz)ds,
\end{align*}
where $y(t):=\Phi(t,z,y_0)$ solves
\[
\frac{dy}{dt}=zf(y),
\]
with initial condition $y(0)=y_0.$

\end{lemma}

The following lemma gives the evolution of $H_1$ and $H_2$ by \eqref{d2-ch}.

\begin{lemma}\label{lm21}
Let $u^\epsilon(t,x)$ be the solution of equation \eqref{d2-ch}.
Then, for $H_1(u^\epsilon), H_2(u^\epsilon)$ given in \eqref{i1} and \eqref{i1a}, we have
\begin{align}
H_1(u^\epsilon)=&\|\varphi_{c_0}\|_{H^1}^2+\int_0^t\int_{\m{Z}}[H_1(\Phi(1,z,u(s-)))-H_1(u^\epsilon(s-,x))]\tilde{\mathcal{N}}(ds,dz)\nonumber\\&
+\int_0^t\int_{\m{Z}}[H_1(\Phi(1,z,u(s-)))-H_1(u^\epsilon(s-,x))+\epsilon z(\sigma_x, u^{\epsilon2}-u^{\epsilon2}_x)]\vartheta(dz)ds,\label{m1xq}\\
H_2(u^\epsilon)=&H_2(\varphi_{c_0})+\int_0^t\int_{\m{Z}}[H_2(\Phi(1,z,u(s-)))-H_2(u^\epsilon(s-,x))]\tilde{\mathcal{N}}(ds,dz)\nonumber\\
&+\int_0^t\int_{\m{Z}}[H_2(\Phi(1,z,u(s)))-H_2(u^\epsilon(s-,x))-\epsilon zH_2'(u^\epsilon)\sigma u^\epsilon_x]\vartheta(dz)ds,\label{m2a}
\end{align}
where $\Phi(1,z,u(s))=u(s,x-\eps\sigma z), H_1'(u^\epsilon)=u^\epsilon-u^\epsilon_{xx}$ and
\[
H_2'(u^\epsilon)=3u^{\epsilon 2}-u_x^{\epsilon2}-2u^\epsilon u^\epsilon_{xx}+4ku^\epsilon.
\]

\begin{proof}
Since the initial value $\vp_c$ is smooth, we can use It\^o formula to $H_1(u^\epsilon)$ and $H_2(u^\epsilon)$.
Since $H_1(u^\epsilon)$ and $H_2(u^\epsilon)$ is the invariants of the CH equation \eqref{11a},
applying It\^o formula Lemma \ref{ito2} to $H_1(u^\epsilon)$ and $H_2(u^\epsilon)$, we have
\eqref{m1xq} and \eqref{m2a}.
\end{proof}

\end{lemma}

\begin{lemma}\label{lm21xx}
Let $u^\epsilon(t,x)$ be the solution of equation \eqref{d2-ch} with $u^\epsilon(0,x)=\varphi_{c_0}(x)$.
Then,
\begin{align}
H_1(\Phi(1,z,u^\epsilon))-H_1(u^\epsilon)\leq& (e^{|z|\|\sigma_x\|_{L^\infty}}-1)\|u^\epsilon\|_{H^1}^2, \label{lmh12-1d}\\
H_2(\Phi(1,z,u^\epsilon))-H_2(u^\epsilon)\leq& C(e^{|z|\|\sigma_x\|_{L^\infty}}-1)\|u^\epsilon\|_{H^1}^2
+C(e^{\frac32|z|\|\sigma_x\|_{L^\infty}}-1)\|u^\epsilon\|_{H^1}^3,\label{lmh12-1ds}
\end{align}
where $\Phi(1,z,u(s))=u(s,x-\eps\sigma z).$ Moreover,
\begin{align}\label{lmh12x-1ds}
\me\sup_{t\in[0,T]}\|u^\epsilon\|_{H^1}^2\leq C  \|\varphi_{c_0}(x)\|_{H^1}^2.
\end{align}

\begin{proof}
Fixed $z\in\m{Z}, s\in\mr$, let $y(t,x)=\Phi(t,z,u(s,x))$ be the solution of the following equation
\[
dy(t,x)=-z\sigma(x) y_x(t,x)dt,\;\; y(0,x)=u(s,x).
\]
Then $y(t,x)=u(s, x-\sigma zt)$ and
\begin{align}\label{yl2}
\|y(t,x)\|_{H^1}^2\leq \|u(s,x)\|_{H^1}^2e^{|z|\|\sigma_x\|_{L^\infty}t}.
\end{align}
By the mean value theorem and $H_1'(y)=y-y_{xx}$,
\begin{align}\label{lmh12-1}
&H_1(\Phi(1,z,u))-H_1(u)=H_1(y(1,x))-H_1(y(0,x))\nonumber\\
=&\int_0^1\frac{d}{dr}[H_1\circ y](r)dr
=\int_0^1(H_1\circ y)'(r)dr\nonumber\\
=&\int_0^1H'_1(y(r)) y'(r)dr=-z\int_0^1H'_1(y(r)) \sigma y_x(r)dr\nonumber\\
=&-z\int_0^1(y-y_{xx}, \sigma y_x)(r)dr
\leq |z|\|\sigma_x\|_{L^\infty}\int_0^1\|y(r)\|_{H^1}^2dr\nonumber\\
\leq& |z|\|\sigma_x\|_{L^\infty}\|u(s,x)\|_{H^1}^2\int_0^1e^{|z|\|\sigma_x\|_{L^\infty}r}dr\nonumber\\
\leq& (e^{|z|\|\sigma_x\|_{L^\infty}}-1)\|u(s,x)\|_{H^1}^2.
\end{align}
Since $H_2'=3y^2-y_x^2-2yy_{xx}+4ky,$ we also have
\begin{align}
&H_2(\Phi(1,z,u))-H_2(u)=-z\int_0^1H'_2(y(r)) \sigma y_x(r)dr\nonumber\\
=&-z\int_0^1(3y^2-y_x^2-2yy_{xx}+4ky, \sigma y_x)(r)dr\nonumber\\
=&z\int_0^1(\sigma_x, y^3-yy_x^2+4ky^2)(r)dr\nonumber\\
\leq& |z|\|\sigma_x\|_{L^\infty}\int_0^1(\|y\|_{L^\infty}\|y\|_{H^1}^2+4k\|y\|_{L^2}^2)dr\nonumber\\
\leq& C|z|\|\sigma_x\|_{L^\infty}\int_0^1(\|y\|_{H^1}^2+\|y\|_{H^1}^3)dr\nonumber\\
\leq& C(e^{|z|\|\sigma_x\|_{L^\infty}}-1)\|u(s,x)\|_{H^1}^2
+C(e^{\frac32|z|\|\sigma_x\|_{L^\infty}}-1)\|u(s,x)\|_{H^1}^3.
\end{align}

Next, we prove \eqref{lmh12x-1ds}.
Using Burkh\"older-Davis-Gundy (BDG) inequality, H\"older inequality and \eqref{lmh12-1d}, we have
\begin{align}\label{yl32-3c2}
&\m{E}\sup_{0\leq t\leq T}|\int_0^t\int_{\m{Z}}[H_1(\Phi(1,z,u^\epsilon))-H_1(u^\epsilon)]\tilde{\mathcal{N}}(ds,dz)|\nonumber\\
\leq &C\m{E}(\int_0^T\int_{\m{Z}}
(e^{|z|\|\sigma_x\|_{L^\infty}}-1)^2\|u^\epsilon(s,x)\|_{H^1}^4\vartheta(dz)ds)^{1/2}\nonumber\\
\leq &C\m{E}\sup_{t\in[0,T]}\|u^\epsilon\|_{H^1}(\int_0^T\int_{\m{Z}}\|u^\epsilon\|_{H^1}^2(e^{|z|\|\sigma_x^n\|_{L^\infty}}-1)^2\vartheta(dz)ds)^{1/2}\nonumber\\
\leq &\frac1{4}\m{E}\sup_{t\in[0,T]}\|u^\epsilon\|_{H^1}^{2}+C\m{E}\int_0^T\|u^\epsilon\|_{H^1}^{2}ds,
\end{align}
Then, it follows from  \eqref{m1xq} and \eqref{yl32-3c2} that
\begin{align*}
 \m{E}\sup_{t\in[0, T]}\|u^\epsilon(t)\|_{H^1}^2
\leq 2\|\varphi_{c_0}(x)\|_{H^1}^2
+ C\m{E}\int_0^{T}\|u^\epsilon(r)\|_{H^1}^2dr,
\end{align*}
from which, the Gronwall inequality  yields \eqref{lmh12x-1ds}.
The proof is complete.
\end{proof}

\end{lemma}

\begin{remark}[Full Euler-Poincar\'e structure in the noise] \label{remark1}
It can be verified that
there is no additional difficulty with the incorporation of full Euler-Poincar\'e noise
of the form $\zeta(u)\diamond dL(t)$ in \eqref{d6ge} in place of $\sigma u_x\diamond dL(t)$ in \eqref{d6a}.

We estimate \eqref{lmh12-1d} with the full Euler-Poincar\'e noise to explain it.
Fixed $z\in\m{Z}, s\in\mr$, let $y(t,x)=\Phi(t,z,u(s,x))$ be the solution of the following equation
\[
dy(t,x)=-z\zeta(y)(t,x)dt,\;\; y(0,x)=u(s,x).
\]
Then, using integration by parts and H\"older inequality,
\begin{align*}
&\frac{d}{dt}\|y(t,x)\|_{H^1}^2=-2z((1-\p_x^2)y, \zeta(y))\\
=&-2z((1-\p_x^2)y, 3\sigma y_x-2\sigma_xy+(1-\p_x^2)^{-1}(2\sigma_xy-y\p_x^3\sigma +\sigma_{xx}y_x)\\&+2k(1-\p_x^2)^{-1}\sigma_x)\\
\leq& C|z|(\|\sigma_x\|_{L^\infty}+\|\sigma_{xx}\|_{L^\infty}+\|\sigma_{xxx}\|_{L^\infty})\|y\|_{H^1}^2+C|z|\|\sigma_x\|_{L^2}\|y\|_{L^2}\\
\leq& C|z|(1+\|\sigma_x\|_{L^\infty}+\|\sigma_{xx}\|_{L^\infty}+\|\sigma_{xxx}\|_{L^\infty})\|y\|_{H^1}^2+C|z|\|\sigma_x\|_{L^2}^2,
\end{align*}
which implies, by Gronwall inequality
\begin{align}\label{rmyl2}
\|y(t,x)\|_{H^1}^2\leq \|u(s,x)\|_{H^1}^2e^{C_1|z|t}+C(e^{C_1|z|t}-1),
\end{align}
where
\[
C_1=C(1+\|\sigma_x\|_{L^\infty}+\|\sigma_{xx}\|_{L^\infty}+\|\sigma_{xxx}\|_{L^\infty}).
\]
By the mean value theorem and $H_1'(y)=y-y_{xx}$,
\begin{align*}
&H_1(\Phi(1,z,u))-H_1(u)=H_1(y(1,x))-H_1(y(0,x))\nonumber\\
=&\int_0^1H'_1(y(r)) y'(r)dr=-z\int_0^1H'_1(y(r)) \zeta(y)(r)dr\nonumber\\
=&-z\int_0^1(y-y_{xx}, \zeta(y))(r)dr\nonumber\\
\leq &C_1|z|\int_0^1\|y(r)\|_{H^1}^2dr+C|z|\int_0^1\|\sigma_x\|_{L^2}^2dr\nonumber\\
\leq& C_1|z|\|u(s,x)\|_{H^1}^2\int_0^1e^{C_1|z|r}dr
+  C_1C|z|\int_0^1(e^{C_1|z|r}-1)dr+C|z| \|\sigma_x\|_{L^2}^2\nonumber\\
\leq& (e^{C_1|z|}-1)\|u(s,x)\|_{H^1}^2+C(e^{C_1|z|}-1)+C|z|.
\end{align*}
Hence, the only extra requirement is that $\sigma\in W^{3,\infty}$ instead of $\sigma\in W^{2,\infty}$.

\end{remark}

Now, we give the proof of Theorem \ref{tm2}.

\begin{proof}[Proof of Theorem \ref{tm2}]
Denote $B(\vp_{c_0}(x),2\alpha)
=\{v\in H^1, \|v(x)-\vp_{c_0}(x)\|_{H^1}\leq2\alpha\}$
for $\alpha>0$. Then, consider a $C^2$ mapping
\begin{align*}
Y: &(c_0-2\alpha, c_0+2\alpha)\times (-2\alpha, 2\alpha)\times B(\vp_{c_0}(x),2\alpha)\to \mr^{2},\\
&(c, x_1, u)\to (Y_1, Y_2)
\end{align*}
defined by
\begin{align*}
&Y_1(c, x_1, u)=\int_\mr(u(x+x_1)-\vp_{c}(x))(1-\p_x^2)\p_x\varphi_{c_0}(x)dx,\\
&Y_2(c, x_1, u)=\int_\mr(u(x+x_1)-\vp_{c}(x))(1-\p_x^2)\varphi_{c_0}(x)dx.
\end{align*}

In the following, we verify that the function $Y$ satisfies the properties:

(i) $Y(c_0, 0, \vp_{c_0}(x))=(0,0).$

(ii)  By the dominated convergence theorem and the smoothness of $\vp_c$, the
partial derivatives $\frac{\p Y_1}{\p c}, \frac{\p Y_1}{\p x_1}, \frac{\p Y_2}{\p c}, \frac{\p Y_2}{\p x_1}$ are continuous. Indeed,
\begin{align*}
\frac{\p Y_1}{\p c}(c_0, 0, \vp_{c_0}(x))=&-\int_\mr\p_{c}\vp_{c}(x)(1-\p_x^2)\p_x\varphi_{c_0}(x)dx|_{(c_0, 0, \vp_{c_0}(x))},\\
\frac{\p Y_2}{\p c}(c_0, 0, \vp_{c_0}(x))=&-\int_\mr\p_{c}\vp_{c}(x)(1-\p_x^2) \varphi_{c_0}(x)dx|_{(c_0, 0, \vp_{c_0}(x))}\\
=&-\int_\mr\p_{c}((\vp_{c_0}(x))^2+(\p_x\vp_{c_0}(x))^2)dx\neq0,\\
\frac{\p Y_1}{\p x_1}(c_0, 0, \vp_{c_0}(x))=&\int_\mr\p_x\vp_{c_0}(x)(1-\p_x^2)\p_x\varphi_{c_0}(x)dx>0,
\end{align*}
and
\begin{align*}
\frac{\p Y_2}{\p x_1}(c_0, 0, \vp_{c_0}(x))=\int_\mr\p_x\vp_{c_0}(x)(1-\p_x^2)\varphi_{c_0}(x)dx=0.
\end{align*}
Hence, the determinant of the matrix
$
Y'_{(c, x_1)}(c_0, 0, \vp_{c_0}(x))\neq0.
$
So, from the implicit function theorem, we find that there exists $\alpha_0 > 0$ and the
uniquely determined $C^2$-functions $(c(u), x_1(u))$ defined for $u\in B(\vp_{c_0}(x),2\alpha)$,
such that
\[
Y(c(u), x_1(u), u)=0.
\]
Moreover, reducing again $\alpha$ if necessary, we may apply the implicit function theorem uniformly around the points $(c,0,u_0)$ satisfying
\[
Y(c, 0, u_0)=0,\;\;|c-c_0|<\alpha,\;\;\hbox{and}\;\;\|u_0-\varphi_{c_0}\|_{H^1}<\alpha.
\]
Applying this with $u = u^\epsilon(t)$, we get the existence of $c^\epsilon(t)=c(u^\epsilon(t))$ and $x^\epsilon(t)=x_1(u^\epsilon(t))$
such that the orthogonality conditions \eqref{orth1}  hold with  $\epsilon\eta^\epsilon(t)=u^\epsilon(t,x+x^\epsilon(t))
-\varphi_{c^\epsilon(t)}(x)$.

Since $u^\epsilon(t)$ is a $H^1$-valued process, it follows that $u^\epsilon(t)$ is a semi-martingale process in $H^{-1}$. Noting
that $Y$ is a $C^1$ functional of $u$ on $H^{-1}$, the processes  $c^\epsilon(t)$ and $x^\epsilon(t)$ are given
locally by a deterministic $C^2$ function of $u^\epsilon(t)\in H^{-1}$. Then the It\^o formula shows that $c^\epsilon(t)$ and $x^\epsilon(t)$ are semi-martingale processes. Moreover, since it is clear that $Y(c^\epsilon(t), x^\epsilon(t), u^\epsilon(t))=0$, the existence of $c^\epsilon(t)$ and $x^\epsilon(t)$ holds as long as
\begin{align}\label{d7}
|c^\epsilon(t)-c_0|<\alpha\;\;\hbox{and}\;\;
\|u^\epsilon(t,x+x^\epsilon(t))
-\varphi_{c^\epsilon(t)}(x)\|_{H^1}<\alpha.
\end{align}

We now define two stopping times
\begin{align*}
&\tilde{\tau}_\alpha^\epsilon=\inf\{t\geq0, |c^\epsilon(t)-c_0|\geq\alpha\;\;\hbox{or}\;\;
 \|u^\epsilon(t,x+x^\epsilon(t))
-\varphi_{c_0}(x)\|_{H^1}\geq\alpha\},\\
&\tau^\epsilon_\beta=\inf\{t\geq0, |c^\epsilon(t)-c_0|\geq\beta\;\;\hbox{or}\;\;
 \|u^\epsilon(t,x+x^\epsilon(t))
-\varphi_{c^\epsilon(t)}(x)\|_{H^1}\geq\beta\}.
\end{align*}
Since the inequality $\|\vp_{c^\epsilon(t)}-\vp_{c_0}\|_{H^1}\leq C\alpha$ holds as long as $|c^\epsilon(t)-c_0|\leq \alpha\leq\alpha_0$, with a
constant $C$ depending only on $\alpha_0$ and $c_0$, it follows obviously that
\[
\tau_\alpha^\epsilon\leq \tilde{\tau}_{(C+1)\alpha}^\epsilon\leq \tau_{(C+1)^2\alpha}^\epsilon.
\]
Taking $\alpha_0$ sufficiently small again, the processes $c^\epsilon(t)$ and $x^\epsilon(t)$
 are defined for all $t\leq \tau^\epsilon_{\alpha_0}$, and
satisfy \eqref{d7} for all $t\leq \tau^\epsilon_{\alpha}, \alpha\leq\alpha_0$ under the orthogonality
conditions \eqref{orth1}.

To prove \eqref{exit} for any $T>0$,
let $H_c=cH_1-H_2$.
By Taylor formula, we have
\begin{align}\label{psi6}
&H_{c_0}(u^\epsilon(t,x+x^\epsilon(t)))-H_{c_0}(\varphi_{c^\epsilon(t)})\nonumber\\
=&(H_{c_0}'(\varphi_{c^\epsilon(t)}), \epsilon\eta^\epsilon(t))+
(H_{c_0}''(\varphi_{c^\epsilon(t)})\epsilon\eta^\epsilon(t), \epsilon\eta^\epsilon(t))
+o(\|\epsilon\eta^\epsilon(t)\|_{H^1}^2).
\end{align}
Note that $o(\|\epsilon\eta^\epsilon(t)\|_{H^1}^2)$ is uniform in $\omega, \epsilon$ and $t$, since $H_{c_0}'(\varphi_{c})$ and $H_{c_0}''(\varphi_{c})$ depend
continuously on $c$, and since $|c^\epsilon(t)-c_0|\leq\alpha\;\;\hbox{and}\;\;
\|u^\epsilon(t,x+x^\epsilon(t))-\varphi_{c^\epsilon(t)}\|_{H^1}\leq\alpha$ for all $t\leq \tau^\epsilon_\alpha$. We then assume $\alpha_0$ small enough so that the last term is less than  $\frac{C}{4}\|\epsilon\eta^\epsilon(t)\|_{H^1}^2$
for all $t\leq \tau^\epsilon_\alpha$.

On account of the results on the spectrum of $\mathcal{L}_{c_0}=H_{c_0}''(\vp_{c^0})$ derived in \cite{constr02} (see also (4.26) in Lemma 4.3 in \cite{dika07}),  there exist $\delta>0$ and
$C_\delta>0$, such that if for $c_0\geq k$,
\[
|(w,(1-\p_x^2)\vp_{c_0})|+|(w,(1-\p_x^2)\p_x\vp_{c_0})|\leq \delta\|w\|_{H^1},
\]
then
\[
(\mathcal{L}_{c_0}w,w)\geq C_\delta\|w\|_{H^1}^2.
\]
It then follows that
\begin{align}\label{psi6a}
&(H_{c_0}''(\varphi_{c^\epsilon(t)})\epsilon\eta^\epsilon(t), \epsilon\eta^\epsilon(t))\nonumber\\
=&(H_{c_0}''(\varphi_{c^\epsilon(t)})\epsilon\eta^\epsilon(t)-H_{c_0}''(\varphi_{c_0})\epsilon\eta^\epsilon(t), \epsilon\eta^\epsilon(t))
+(H_{c_0}''(\varphi_{c_0})\epsilon\eta^\epsilon(t), \epsilon\eta^\epsilon(t))\nonumber\\
\geq&C_\delta \|\epsilon\eta^{\epsilon }\|_{H^1}^2-\|H_{c_0}''(\varphi_{c^\epsilon(t)})-H_{c_0}''(\varphi_{c_0})\|_{\mathcal{L}(H^1, H^{-2})}\|\epsilon\eta^{\epsilon }\|_{H^1}^2\nonumber\\
\geq&C_\delta \|\epsilon\eta^{\epsilon }\|_{H^1}^2-C|c^\epsilon(t)-c_0|\|\epsilon\eta^{\epsilon }\|_{H^1}^2.
\end{align}
Since $H_{c^\epsilon}'(\varphi_{c^\epsilon})=0$ by \eqref{i4}, H\"older and Young inequalities yield
\begin{align}\label{psi7a}
|(H_{c_0}'(\varphi_{c^\epsilon(t)}), \epsilon\eta^\epsilon(t))|
=&|(H_{c_0}'(\varphi_{c^\epsilon})-H_{c^\epsilon}'(\varphi_{c^\epsilon}), \epsilon\eta^\epsilon(t))\nonumber\\
=&2((c^\epsilon-c_0)\varphi_{c^\epsilon},\epsilon\ee)|\nonumber\\
\leq &C|c^\epsilon(t)-c_0|^2+\frac{C_\delta}2\|\epsilon\eta^{\epsilon }\|_{H^1}^2.
\end{align}
It follows from \eqref{psi6}-\eqref{psi7a} that for all $\alpha\leq\alpha_0\leq \frac{C_\delta}2$,
and for all $t\leq \tau_\alpha^\epsilon$, we have
\begin{align*}
H_{c_0}(u^\epsilon(t,x+x^\epsilon(t)))-H_{c_0}(\varphi_{c^\epsilon(t)})\geq \frac{C_\delta}4\|\epsilon\eta^{\epsilon }(t)\|_{H^1}^2-C|c^\epsilon(t)-c_0|^2,
\end{align*}
for a constant C depending only on $\alpha_0$ and $c_0.$ Hence, using It\^o formula to $H_{c_0}(u^\epsilon(t,x))$, we have
\begin{align}\label{psi8}
&\|\epsilon\eta^{\epsilon }(t, x-x^\epsilon(t))\|_{H^1}^2\leq C[H_{c_0}(u^\epsilon(t,x))-H_{c_0}(\varphi_{c^\epsilon(t)}(x-x^\epsilon(t)))]
+C|c^\epsilon(t)-c_0|^2\nonumber\\
=& C\{H_{c_0}(\varphi_{c_0})-H_{c_0}(\varphi_{c^\epsilon(t)}(x-x^\epsilon))\nonumber\\&
+\int_0^t\int_{\m{Z}}[H_{c_0}(\Phi(1,z, u^\epsilon(s-,x))-H_{c_0}(u^\epsilon(s-,x))]\tilde{\mathcal{N}}(ds,dz)\nonumber\\
&+\int_0^t\int_{\m{Z}}[H_{c_0}(\Phi(1,z, u^\epsilon(s,x))-H_{c_0}(u^\epsilon(s,x))-\epsilon zH_{c_0}'(u^\epsilon)\sigma u^\epsilon_x]\vartheta(dz)ds\}\nonumber\\&
+C|c^\epsilon(t)-c_0|^2,
\end{align}
where
\[
H_{c_0}'(u)=c_0H_1'(u)-H_2'(u)=c_0(u-u_{xx})
-(3u^2-u_x^2-2uu_{xx}+4ku).
\]

We now estimate $|c^\epsilon(t)-c_0|^2$. The orthogonality conditions \eqref{orth1} imply
\begin{align}\label{cx1}
\|u^\epsilon(t)\|_{H^1}^2=&\|\epsilon\eta^\epsilon+\varphi_{c^\epsilon(t)}\|_{H^1}^2\nonumber\\
=&\|\epsilon\eta^\epsilon\|_{H^1}^2+\|\varphi_{c^\epsilon(t)}\|_{H^1}^2
+2(\epsilon\eta^\epsilon, (1-\p_x^2)(\varphi_{c^\epsilon(t)}-\varphi_{c_0}))
\end{align}
and by \eqref{m1xq},
\begin{align}\label{cx2}
\|u^\epsilon(t)\|_{H^1}^2=&\|\varphi_{c_0}\|_{H^1}^2+\int_0^t\int_{\m{Z}}G_1(u^\epsilon(s-,x))\tilde{\mathcal{N}}(ds,dz)\nonumber\\&
+\int_0^t\int_{\m{Z}}[G_1(u^\epsilon(s,x))+\eps z(\sigma_x, u^{\epsilon2}-u^{\epsilon2}_x)]\vartheta(dz)ds,
\end{align}
where
\[
G_1(u^\epsilon(s,x))=\|\Phi(1,z, u^\epsilon(s,x)\|_{H^1}^2-\|u^\epsilon(s,x)\|_{H^1}^2.
\]
Thus, it follows from \eqref{cx1}-\eqref{cx2} that for some constants $C$ and $\mu,$ depending only on $\alpha_0$ and $c_0$,
\begin{align}\label{cx3}
&\mu|c^\epsilon(t)-c_0|\nonumber\\
\leq& | \|\varphi_{c^\epsilon(t)}\|_{H^1}^2-\|\varphi_{c_0}\|_{H^1}^2| \nonumber\\
\leq& \|\epsilon\eta^\epsilon\|_{H^1}^2+2\|\epsilon\eta^\epsilon\|_{H^1}(\|\varphi_{c^\epsilon(t)}-\varphi_{c_0}\|_{H^1})
+\int_0^t\int_{\m{Z}}G_1(u^\epsilon(s-,x))\tilde{\mathcal{N}}(ds,dz)\nonumber\\&
+\int_0^t\int_{\m{Z}}[G_1(u^\epsilon(s,x))+\eps z(\sigma_x, u^{\epsilon2}-u^{\epsilon2}_x)]\vartheta(dz)ds
\nonumber\\
\leq& \|\epsilon\eta^\epsilon\|_{H^1}^2+C\alpha|c^\epsilon(t)-c_0|
+\int_0^t\int_{\m{Z}}G_1(u^\epsilon(s-,x))\tilde{\mathcal{N}}(ds,dz)\nonumber\\&
+\int_0^t\int_{\m{Z}}[G_1(u^\epsilon(s,x))+\eps z(\sigma_x, u^{\epsilon2}-u^{\epsilon2}_x)]\vartheta(dz)ds.
\end{align}
Hence, choosing $\alpha_0$ sufficient small, we get
\begin{align}\label{cx4}
|c^\epsilon(t)-c_0|^2
\leq& C[\|\epsilon\eta^\epsilon\|_{H^1}^4+
|\int_0^t\int_{\m{Z}}G_1(u^\epsilon(s-,x))\tilde{\mathcal{N}}(ds,dz)|^2\nonumber\\&
+|\int_0^t\int_{\m{Z}}[G_1(u^\epsilon(s,x))+\eps z(\sigma_x, u^{\epsilon2}-u^{\epsilon2}_x)]\vartheta(dz)ds|^2].
\end{align}
Because $H'_{c_0}(\varphi_{c_0})=0,$ we have
\begin{align}\label{cx4a}
|H_{c_0}(\varphi_{c_0})-H_{c_0}(\varphi_{c^\epsilon(t)})|
\leq C\|\varphi_{c_0}-\varphi_{c^\epsilon(t)}\|_{H^1}^2\leq
C|c^\epsilon(t)-c_0|^2.
\end{align}
Then, inserting \eqref{cx4}-\eqref{cx4a} in the right hand of \eqref{psi8}, we obtain
\begin{align}\label{psi12}
&\|\epsilon\eta^{\epsilon }(t)\|_{H^1}^2
\leq C\{\|\epsilon\eta^\epsilon\|_{H^1}^4+|\int_0^t\int_{\m{Z}}G_1(u^\epsilon(s-,x))\tilde{\mathcal{N}}(ds,dz)|^2\nonumber\\&
+|\int_0^t\int_{\m{Z}}[G_1(u^\epsilon(s,x))+\eps z(\sigma_x, u^{\epsilon2}-u^{\epsilon2}_x)]\vartheta(dz)ds|^2\nonumber\\
&
+\int_0^t\int_{\m{Z}}[H_{c_0}(\Phi(1,z,u^\epsilon(s-,x))-H_{c_0}(u^\epsilon(s-,x))]\tilde{\mathcal{N}}(ds,dz)\nonumber\\
&+\int_0^t\int_{\m{Z}}[H_{c_0}(\Phi(1,z,u^\epsilon(s,x))-H_{c_0}(u^\epsilon(s,x))-\epsilon zH_{c_0}'(u^\epsilon)\sigma u^\epsilon_x]\vartheta(dz)ds\}.
\end{align}

Now, fix $T>0$ and set
\begin{align*}
&\Omega_1^{T, \epsilon, \alpha}=\{\omega\in \Omega, \tau_\alpha^\epsilon\leq T, \|\epsilon\eta^\epsilon(\tau_\alpha^\epsilon)\|_{H^1}=\alpha\},\\
&\Omega_2^{T, \epsilon, \alpha}=\{\omega\in \Omega, \tau_\alpha^\epsilon\leq T, |c^\epsilon(\tau_\alpha^\epsilon)-c_0|=\alpha  \}
\end{align*}
 so that
\begin{align*}
\m{P}(\tau_\alpha^\epsilon\leq T)\leq \m{P}(\Omega_1^{T, \epsilon, \alpha})+\m{P}(\Omega_2^{T, \epsilon, \alpha}).
\end{align*}
Let $\alpha_0>0$ be small enough so that $C\alpha_0^2\leq1/2$.
Multiplying both sides of \eqref{psi12}
by $1_{\Omega_1^{T, \epsilon, \alpha}}$, for $\alpha\leq\alpha_0$ and taking expectation with $t=\tau_\alpha^\epsilon\wedge T$, we have
\begin{align*}
&\frac{\alpha^2}2\m{P}(\Omega_1^{T, \epsilon, \alpha})\nonumber\\
\leq&
C\{\me|\int_0^t\int_{\m{Z}}G_1(u^\epsilon(s-,x))\tilde{\mathcal{N}}(ds,dz)1_{\Omega_1^{T, \epsilon, \alpha}}|^2\nonumber\\&
+\me|\int_0^t\int_{\m{Z}}[G_1(u^\epsilon(s,x))+\eps z(\sigma_x, u^{\epsilon2}-u^{\epsilon2}_x)]\vartheta(dz)ds1_{\Omega_1^{T, \epsilon, \alpha}}|^2\nonumber\\
&
+\me[\int_0^t\int_{\m{Z}}[H_{c_0}(\Phi(1,z,u^\epsilon(s-,x))-H_{c_0}(u^\epsilon(s-,x))]\tilde{\mathcal{N}}(ds,dz)1_{\Omega_1^{T, \epsilon, \alpha}}]\nonumber\\
&+\me[\int_0^t\int_{\m{Z}}[H_{c_0}(\Phi(1,z,u^\epsilon(s,x))-H_{c_0}(u^\epsilon(s,x))-\epsilon zH_{c_0}'(u^\epsilon)\sigma u^\epsilon_x]\vartheta(dz)ds1_{\Omega_1^{T, \epsilon, \alpha}}]\}.
\end{align*}
Using Lemma \ref{lm21xx},
\begin{align*}
&G_1(u^\epsilon(s,x))
\leq  C(e^{\eps|z|\|\sigma_x\|_{L^\infty}}-1)\|u(s,x)\|_{H^1}^2,\\
&H_{c_0}(\Phi(1,z,u^\epsilon(s,x))-H_{c_0}(u^\epsilon(s,x))\\
\leq &C(e^{\eps|z|\|\sigma_x\|_{L^\infty}}-1)\|u^\eps(s,x)\|_{H^1}^2
+C(e^{\frac32\eps|z|\|\sigma_x\|_{L^\infty}}-1)\|u^\eps(s,x)\|_{H^1}^3,\\
&H'_{c_0}(u^\epsilon)\sigma u_x\leq C(\|u^\eps\|_{H^1}^2+\|u^\eps\|_{H^1}^3),
\end{align*}
Then, by Cauchy inequality, BDG inequality and $\|u^\epsilon\|_{H^1}^2\leq C$ a.s., we can get
\begin{align*}
&\frac{\alpha^2}2\m{P}(\Omega_1^{T, \epsilon, \alpha})\nonumber\\
\leq&CT\int_{\m{Z}}(e^{\eps|z|\|\sigma_x\|_{L^\infty}}-1)^2\vartheta(dz)\m{P}(\Omega_1^{T, \epsilon, \alpha})\\&
+CT^2[|\int_{\m{Z}}(e^{\eps|z|\|\sigma_x\|_{L^\infty}}-1)\vartheta(dz)|+\epsilon^2\|\sigma_x\|_{L^\infty}^2|\int_{\m{Z}}z\vartheta(dz)|]^2\m{P}(\Omega_1^{T, \epsilon, \alpha})\nonumber\\&
+C\sqrt{T}(\int_{\m{Z}}(e^{\eps|z|\|\sigma_x\|_{L^\infty}}-1)^2+(e^{\frac32\eps|z|\|\sigma_x\|_{L^\infty}}-1)^2\vartheta(dz))^{1/2}\m{P}(\Omega_1^{T, \epsilon, \alpha})^{1/2},
\end{align*}
and it follows that, for $\epsilon$ sufficient small
\begin{align*}
\m{P}(\Omega_1^{T, \epsilon, \alpha})\leq C\frac{\sqrt{T}}{\alpha^2}b^{1/2}(\epsilon)\m{P}(\Omega_1^{T, \epsilon, \alpha})^{1/2},
\end{align*}
where $b(\epsilon)=\int_{\m{Z}}((e^{\eps|z|\|\sigma_x\|_{L^\infty}}-1)^2+(e^{\frac32\eps|z|\|\sigma_x\|_{L^\infty}}-1)^2)\vartheta(dz),$
which implies
\begin{align}\label{p1-a}
\m{P}(\Omega_1^{T, \epsilon, \alpha})\leq C\frac{T}{\alpha^4}b(\epsilon).
\end{align}
Coming back to \eqref{cx4} and using the same argument as above, we can obtain
\[
\alpha^2\m{P}(\Omega_2^{T, \epsilon, \alpha})\leq C(\eps)\m{P}(\Omega_2^{T, \epsilon, \alpha})
+C\sqrt{T}b^{1/2}(\epsilon)\m{P}(\Omega_2^{T, \epsilon, \alpha})^{1/2},
\]
then, for $\eps$ sufficient small, we have
\begin{align}\label{p1-b}
\m{P}(\Omega_2^{T, \epsilon, \alpha})\leq C\frac{T}{\alpha^4}b(\epsilon).
\end{align}
Hence, \eqref{exit} follows from \eqref{p1-a}-\eqref{p1-b}  for $\alpha$ and $\epsilon$ sufficient small.
\end{proof}

\smallskip



\section{Modulation equations}\label{s3}

In this section, we derive the equation coupling the modulation parameters $x^\epsilon, c^\epsilon$ to the remaining term $\ee.$

\begin{lemma}
Under the assumptions of Theorem \ref{th3}, $\eta^\epsilon$ satisfies the equation
\begin{align}\label{me1}
d\eta^\epsilon=&\frac12(1-\p_x^2)^{-1}\p_x\mathcal{L}_{c^\epsilon}\eta^\epsilon dt +( y^\epsilon \p_x\vp_{c^\epsilon}-\p_c\vp_{c^\epsilon}a^\epsilon) dt
+\epsilon y^\epsilon \ee_x dt+\epsilon f(\ee)dt \nonumber\\
&
+[(\sigma\p_x\vp_{c^\epsilon}+\p_x\vp_{c^\epsilon}  \mu^\epsilon)-\p_c\vp_{c^\epsilon}b^\epsilon
+\epsilon (\sigma\ee_x+\ee_x\mu^\epsilon)]\diamond dL(t),
\end{align}
where
\begin{align}\label{me2}
&f(\ee)=-\ee\eta^\epsilon_x-(1-\p_x^2)^{-1}\p_x(\eta^{\epsilon2}+\frac12\eta^{\epsilon2}_x).
\end{align}

\begin{proof}
We write \eqref{d2-ch} in the Hamiltonian form
\[
du+(1-\p_x^2)^{-1}\p_xH_2'(u)dt+\epsilon \sigma u_x\diamond dL(t).
\]
Then, using \eqref{xe1}-\eqref{ce1}, we have
\begin{align}
&du(t,x+x^\epsilon)=-(1-\p_x^2)^{-1}\p_xH_2'(u)(t,x+x^\epsilon)dt+\epsilon \sigma u_x(t,x+x^\epsilon)\diamond dL(t)\nonumber\\
&+u_x(t,x+x^\epsilon)c^\epsilon dt+\epsilon u_x(t,x+x^\epsilon) y^\epsilon dt+\epsilon u_x(t,x+x^\epsilon) \mu^\epsilon\diamond dL(t),\label{me3}\\
&d\vp_{c^\epsilon}=\epsilon\p_c\vp_{c^\epsilon}a^\epsilon dt+\epsilon\p_c\vp_{c^\epsilon}b^\epsilon \diamond dL(t).\label{me4}
\end{align}
Replacing $u(t,x+x^\epsilon)$ by $\vp_{c^\epsilon}(x)+\epsilon\ee(t,x)$ and using \eqref{i4}-\eqref{i5},
\[
H_{2}'(\vp_{c^\epsilon})=c^\epsilon H_1'(\vp_{c^\epsilon})=c^\epsilon(\vp_{c^\epsilon}-\p_x^2\vp_{c^\epsilon}),
\]
we have
\begin{align*}
&H_2'(u)(t,x+x^\epsilon)=(\frac32u^2-\frac12u_x^2-uu_{xx}+2ku)(t,x+x^\epsilon)\nonumber\\
=&H_2'(\vp_{c^\epsilon})+H_2'(\ee)+3\epsilon\ee \vp_{c^\epsilon}-\epsilon\ee\p_x^2\vp_{c^\epsilon}-\epsilon\ee_x\p_x\vp_{c^\epsilon}-\ee_{xx}\vp_{c^\epsilon}\nonumber\\
=&c^\epsilon(\vp_{c^\epsilon}-\p_x^2\vp_{c^\epsilon})-\frac12\mathcal{L}_{c^\epsilon}(\epsilon\ee)-\epsilon c^\epsilon\ee_{xx}+\epsilon c^\epsilon\ee+\frac32\epsilon^2\eta^{\epsilon2}-\frac12\epsilon^2\eta^{\epsilon2}_x-\epsilon^2\ee\ee_{xx}.
\end{align*}
Hence,
\begin{align}\label{me5}
&-(1-\p_x^2)^{-1}\p_xH_2'(u)(t,x+x^\epsilon)\nonumber\\
=&\frac\epsilon2(1-\p_x^2)^{-1}\p_x\mathcal{L}_{c^\epsilon}(\ee)
-c^\epsilon\p_x\vp_{c^\epsilon}-\epsilon c^\epsilon\ee_x+\epsilon^2f(\ee),
\end{align}
where $f(\ee)$ is given by \eqref{me2}.
Applying $(1-\p_x^2)^{-1}$ to both sides of \eqref{me3}, then replacing $u(t,x+x^\epsilon)$ by $\vp_{c^\epsilon}(x)+\epsilon\ee(t,x)$
and putting \eqref{me4}-\eqref{me5} into \eqref{me3}, we get \eqref{me1}.
\end{proof}

\end{lemma}

\begin{lemma}
Under the assumptions of Theorem \ref{th3}, the modulation parameters satisfy the system of the equation
\begin{align}
&A^\epsilon(t) B^\epsilon(t)=D^\epsilon(t),\label{me7}\\
&A^\epsilon(t) Y^\epsilon(t)=E^\epsilon(t)\label{me8},
\end{align}
where

\begin{align}\label{me9}
A^\epsilon(t)=
{\left( \begin{array}{ccc}
   (\p_x\vp_{c^\epsilon}+\epsilon\ee_x, (1-\p_x^2)\p_x\vp_{c_0})& -(\p_c\vp_{c^\epsilon}, (1-\p_x^2)\p_x\vp_{c_0})\\
   (\p_x\vp_{c^\epsilon}, (1-\p_x^2)\vp_{c_0})& -(\p_c\vp_{c^\epsilon}, (1-\p_x^2)\vp_{c_0})
  \end{array}
\right )},
\end{align}
\begin{align}\label{me6}
B^\epsilon(t)=
{\left( \begin{array}{ccc}
    \mu^\epsilon(t) \\
   b^\epsilon(t)
  \end{array}
\right )},\;\;
  Y^\epsilon(t)=
{\left( \begin{array}{ccc}
    y^\epsilon(t) \\
   a^\epsilon(t)
\end{array}
\right )},
\end{align}
\begin{align}\label{me10}
D^\epsilon(t)=
{\left( \begin{array}{ccc}
   -(\sigma\p_x\vp_{c^\epsilon}+\eps\sigma\ee_x, (1-\p_x^2)\p_x\vp_{c_0}))\\
   -(\sigma\p_x\vp_{c^\epsilon}+\eps\sigma\ee_x, (1-\p_x^2)\vp_{c_0})
  \end{array}
\right )},
\end{align}
and
\begin{align}\label{me11}
  E^\epsilon(t)=
{\left( \begin{array}{ccc}
 (-\frac12\p_x\mathcal{L}_{c^\epsilon}\eta^\epsilon, \p_x\vp_{c_0})-\epsilon (f(\ee), (1-\p_x^2)\p_x\vp_{c_0})\\
    (-\frac12\p_x\mathcal{L}_{c^\epsilon}\eta^\epsilon, \vp_{c_0})-\epsilon (f(\ee), (1-\p_x^2)\vp_{c_0})
\end{array}
\right )}.
\end{align}

\begin{proof}
Taking inner product of Eq. \eqref{me1} with $(1-\p_x^2)\vp_{c_0}$ and  $(1-\p_x^2)\p_x\vp_{c_0}$ respectively,
and making the orthogonality conditions \eqref{orth1}, we have
\begin{align*}
&0=d(\eta^\epsilon, (1-\p_x^2)\vp_{c_0})=(d\eta^\epsilon, (1-\p_x^2)\vp_{c_0})\nonumber\\
=&(\p_x\vp_{c^\epsilon}, (1-\p_x^2)\vp_{c_0}) y^\epsilon dt-(\p_c\vp_{c^\epsilon}, (1-\p_x^2)\vp_{c_0})a^\epsilon dt
+
(\frac12\p_x\mathcal{L}_{c^\epsilon}\eta^\epsilon, \vp_{c_0})dt \\
&
+\epsilon (f(\ee),  (1-\p_x^2)\vp_{c_0})dt -(\p_c\vp_{c^\epsilon}, (1-\p_x^2)\vp_{c_0})b^\epsilon\diamond dL(t)\nonumber\\
&
+[ \sigma\p_x\vp_{c^\epsilon}+\epsilon\sigma\ee_x, (1-\p_x^2)\vp_{c_0})
+(\p_x\vp_{c^\epsilon}, (1-\p_x^2)\vp_{c_0})  \mu^\epsilon ]\diamond dL(t),
\end{align*}
and
\begin{align*}
0=&d(\eta^\epsilon, (1-\p_x^2)\p_x\vp_{c_0})=(d\eta^\epsilon, (1-\p_x^2)\p_x\vp_{c_0})\nonumber\\
=&(\p_x\vp_{c^\epsilon}, (1-\p_x^2)\p_x\vp_{c_0}) y^\epsilon dt-(\p_c\vp_{c^\epsilon}, (1-\p_x^2)\p_x\vp_{c_0})a^\epsilon dt
+
(\frac12\p_x\mathcal{L}_{c^\epsilon}\eta^\epsilon, \p_x\vp_{c_0})\nonumber\\
&
+\epsilon (f(\ee), (1-\p_x^2)\p_x\vp_{c_0})dt
+[ \sigma\p_x\vp_{c^\epsilon}+\epsilon\sigma\ee_x, (1-\p_x^2)\p_x\vp_{c_0})\\
&
+(\p_x\vp_{c^\epsilon}+\eps\ee_x, (1-\p_x^2)\p_x\vp_{c_0})  \mu^\epsilon ]\diamond dL(t)-(\p_c\vp_{c^\epsilon}, (1-\p_x^2)\p_x\vp_{c_0})b^\epsilon\diamond dL(t).
\end{align*}
Then, letting both the drift and martingale part of the above equations are identically equal to zero, we get \eqref{me7}-\eqref{me8}.
\end{proof}

\end{lemma}

\begin{lemma}\label{lm43}
Under the assumptions of Theorem \ref{th3}, there is a $\alpha_1$ such that for $\alpha\leq \alpha_1$
and $t\leq \tau_\alpha^\epsilon,$
\begin{align}
&| \mu^\epsilon(t)|+|b^\epsilon(t)|\leq C(c_0, \alpha),\label{me11}\\
&| y^\epsilon(t)|+|a^\epsilon(t)|\leq C(c_0, \alpha)\|\ee\|_{H^1}^2,\;a.s.\label{me12}
\end{align}

\begin{proof}
We may write almost surely for $t\leq \tau^\epsilon_\alpha$ that $A^\epsilon(t)= A_0+O(|c^\epsilon-c_0|+\|\epsilon\ee\|_{H^1}),$ where
\begin{align*}
A_0=
{\left( \begin{array}{ccc}
 (\p_x\vp_{c_0}, (1-\p_x^2)\p_x\vp_{c_0}) &0\\
    0 & (\p_c\vp_{c_0}, (1-\p_x^2)\vp_{c_0})
  \end{array}
\right )}
\end{align*}
and $O(|c^\epsilon-c_0|+\|\epsilon\ee\|_{H^1})$ holds uniformly in $\epsilon, t$ and $\omega$ as long as $t\leq \tau^\epsilon$.
Hence, choosing $\alpha\leq \alpha_1$ smaller, it follows that setting
\[
\tilde{A}^\epsilon(t)=A_0+1_{[0,\tau^\epsilon)}(t)(A^\epsilon(t)-A_0),
\]
the matrix $\tilde{A}^\epsilon(t)$ is invertible and
\[
\|(\tilde{A}^\epsilon(t))^{-1}\|_{\mathcal{L}(\mr^2)}\leq C(c_0, \alpha),\;a.s.
\]
Then, Eq. \eqref{me7} may be solved as
\[
B^\epsilon(t)=(\tilde{A}^\epsilon(t))^{-1}D^\epsilon(t),
\]
for $t\leq \tau^\epsilon$,  which implies
\begin{align*}
&| \mu^\epsilon(t)|+|b^\epsilon(t)|\leq C(c_0, \alpha)|D^\epsilon(t)|
\leq C(c_0, \alpha)(\|\vp_{c^\epsilon}\|_{H^1}+\|\vp_{c_0}\|_{H^3})\leq C(c_0, \alpha).
\end{align*}
Thus, \eqref{me11} is obtained. By \eqref{i5} and \eqref{me2}, for $t\leq \tau^\epsilon$
\begin{align*}
|E^\epsilon(t)|\leq C\|\ee\|_{L^2}\|\mathcal{L}_{c^\epsilon}(\p_x\vp_{c_0})\|_{L^2}+C\eps\|f(\ee)\|_{L^1}\|(1-\p_x^2)\p_x\vp_{c_0}\|_{L^\infty}\\
(1+|c^\epsilon| +\|\ee\|_{H^1}^2), \;a.s.
\end{align*}
Using \eqref{ce1} and \eqref{me11}, we can get
\[
| y^\epsilon(t)|+|a^\epsilon(t)|\leq C(c_0, \alpha)(1+\int_0^t |a^\epsilon(s)|ds+\|\ee\|_{H^1}^2),\;\; a.s.
\]
Then the Gronwall inequality implies \eqref{me12}.
\end{proof}

\end{lemma}

\begin{corollary}\label{co41}
Under the assumptions of Theorem \ref{th3}, for  $t\leq \tau^\epsilon_\alpha$, we have
\begin{align}\label{me11a}
 |c^\epsilon-c_0|\leq C\epsilon (\int_0^T\|\ee\|_{H^1}^2ds+1),\;\; \hbox{a.s.}
\end{align}

\begin{proof}
Since
\[
c^\epsilon-c_0=\epsilon\int_0^t a^\epsilon(s) ds+\epsilon\int_0^t  b^\epsilon(t)\diamond dL(t),
\]
\eqref{me11a} is obtained by Lemma \ref{lm43}.
\end{proof}
\end{corollary}

\smallskip

\section{Estimates on the remainder term and convergence}\label{s4}

In this section, we prove the convergence of $\ee$. Firstly, we give some estimates of $\ee,  y^\epsilon,  \mu^\epsilon, a^\epsilon$ and $b^\epsilon$
in the following Lemmas \ref{lemma51}-\ref{lemma52}.

\begin{lemma}\label{lemma51}
Let $T>0$ be fixed. Under the assumption of Theorem \ref{th3}, we have
\begin{align}
&\m{E}\sup_{t\in[0,T\wedge \tau^\epsilon_\alpha]}\|\ee\|_{H^1}^{2r}\leq C,\;\; r=1, 2\label{ec6a}.
\end{align}

\begin{proof}
Let $y(t,x)=\Phi_1(t,z,v)$ solves the following partial differential equation
\[
\frac{dy}{dt}=z(\sigma\p_x\vp_{c^\epsilon}+\p_x\vp_{c^\epsilon}  \mu^\epsilon)-z\p_c\vp_{c^\epsilon}b^\epsilon
+\epsilon z(\sigma + \mu^\epsilon)y_x,
\]
with initial value $y(0,x)=v(x).$
Then, using integrations by parts, H\"older inequality and \eqref{me11}, we have
\begin{align*}
&\|y(t,x)\|_{H^1}^2=\|y(0,x)\|_{H^1}^2+2z\int_0^t((\sigma\p_x\vp_{c^\epsilon}+\p_x\vp_{c^\epsilon}  \mu^\epsilon)-\p_c\vp_{c^\epsilon}b^\epsilon, (1-\p_x^2)y)ds\\
&+2\eps z\int_0^t((\sigma + \mu^\epsilon)y_x, (1-\p_x^2)y)ds\\
\leq&\|v(x)\|_{H^1}^2+C|z|\int_0^t(\|\sigma_x\|_{L^\infty}\|\p_x\vp_{c^\epsilon}\|_{H^1}+\|\p_x\vp_{c^\epsilon}\|_{H^1} |\mu^\epsilon|
+\|\p_c\vp_{c^\epsilon}\|_{H^1}|b^\epsilon|)\|y\|_{H^1}ds\\
&+C\eps |z|\int_0^t\|\sigma_x\|_{L^\infty}\|y(t,x)\|_{H^1}^2ds\\
\leq&\|v(x)\|_{H^1}^2+C|z|\int_0^t(1+\|y\|_{H^1}^2)ds.
\end{align*}
The Gronwall's inequality implies
\begin{align*}
\|y(t,x)\|_{H^1}^2\leq(\|v(x)\|_{H^1}^2+C|z|t)e^{C|z|t}.
\end{align*}
Similar to the estimate \eqref{lmh12-1}, using integrations by parts, H\"older inequality and \eqref{me11}, we have
\begin{align}\label{lemma51a-1}
&H_1(\Phi_1(1,z,\ee))-H_1(\ee)=\int_0^1H'_1(y(r)) y'(r)dr\nonumber\\
=&z\int_0^1(y-y_{xx}, \sigma\p_x\vp_{c^\epsilon}+\p_x\vp_{c^\epsilon}  \mu^\epsilon-\p_c\vp_{c^\epsilon}b^\epsilon
+\epsilon (\sigma + \mu^\epsilon)y_x)(r)dr\nonumber\\
\leq&C |z|\int_0^1(\|y(r)\|_{H^1}+\|y(r)\|_{H^1}^2)dr\nonumber\\
\leq& C|z|\|\ee(s,x)\|_{H^1}^2\int_0^1e^{C|z|r}dr+C|z|^2\int_0^1 re^{C|z|r}dr\nonumber\\
\leq& C|z|(\|\ee(s,x)\|_{H^1}^2+1).
\end{align}

Using It\^o formula Lemma \ref{ito2} to $\|\ee\|_{H^1}^2$ of \eqref{me1}, we have
\begin{align}\label{lemma51-1}
\|\ee\|_{H^1}^2
=&\int_0^t((1-\p_x^2)^{-1}\p_x\mathcal{L}_{c^\epsilon}\ee, (1-\p^2_x)\ee)ds\nonumber\\
&+2\int_0^t(( y^\epsilon \p_x\vp_{c^\epsilon}-\p_c\vp_{c^\epsilon}a^\epsilon), (1-\p^2_x)\ee)ds
\nonumber\\
&+2\int_0^t(\epsilon y^\epsilon \p_x\ee +\epsilon f(\ee), (1-\p^2_x)\ee)ds\nonumber\\&
+\int_0^t\int_{\m{Z}}
[\|\Phi_1(1,z,\ee(s-))\|_{H^1}^2-\|\ee(s-)\|_{H^1}^2]\tilde{\mathcal{N}}(ds,dz)\nonumber\\
&+\int_0^t\int_{\m{Z}}
[\|\Phi_1(1,z,\ee(s))\|_{H^1}^2-\|\ee\|_{H^1}^2\nonumber\\
&-2z((1-\p^2_x)\ee, (\sigma\p_x\vp_{c^\epsilon}+\p_x\vp_{c^\epsilon}  \mu^\epsilon)-\p_c\vp_{c^\epsilon}b^\epsilon
+\epsilon (\sigma\ee_x+\ee_x\mu^\epsilon))]\vartheta(dz)ds\nonumber\\
=:&\sum_{i=1}^5J_i.
\end{align}
Using \eqref{i5}, integration by parts, Cauchy inequality and embedding theorem, we have
\begin{align}\label{lemma51-2}
J_1=&\int_0^t(\p_x\mathcal{L}_{c^\epsilon}\ee, \ee)ds\nonumber\\
=&2\int_0^t(\p_x^2((c^\epsilon-\vp_{c^\epsilon})\p_x\ee), \ee)ds
+2\int_0^t(\p_x((-3\vp_{c^\epsilon}+\p_x^2\vp_{c^\epsilon}+(c^\epsilon-2k))\ee), \ee)ds\nonumber\\
=&\int_0^t(\p_x\vp_{c^\epsilon}, (\p_x\ee)^2)ds
-2\int_0^t((-3\vp_{c^\epsilon}+\p_x^2\vp_{c^\epsilon})\ee, \p_x\ee)ds\nonumber\\
\leq&C\int_0^t\|\ee\|_{H^1}^2ds.
\end{align}
Lemma \ref{lm43} and Cauchy inequality yield
\begin{align}\label{lemma51-3}
J_2\leq C\int_0^t\|\ee\|_{H^1}^2ds.
\end{align}
Using integration by parts, we obtain
\begin{align}\label{lemma51-4}
J_3=&2\epsilon\int_0^t( (1-\p^2_x)f(\ee), \ee)ds\nonumber\\
=&2\epsilon\int_0^t(3\ee\p_x\ee-2\p_x\ee\p_x^2\ee-\ee\p_x^3\ee, \ee)ds=0.
\end{align}
By BDG inequality and \eqref{lemma51a-1}, we have
\begin{align}\label{lemma51-5}
\m{E}\sup_{t\in[0,T\wedge \tau^\epsilon_\alpha]}J_4
\leq &C\m{E}\int_0^T\left(\int_{\m{Z}}
|z|^2(\|\ee(s,x)\|_{H^1}^2+1)^2\vartheta(dz)\right)^{1/2}ds\nonumber\\
\leq &\frac12\me\sup_{t\in[0,T]}\|\ee\|_{H^1}^2+C\m{E}\int_0^T(\|\ee\|_{H^1}^2+1)ds.
\end{align}
Similarly, we have
\begin{align}\label{lemma51-6}
&\m{E}\sup_{t\in[0,T\wedge \tau^\epsilon_\alpha]}J_5
\leq C(1+\m{E}\int_0^T \|\ee\|_{H^1}^2ds).
\end{align}
It follows from \eqref{lemma51-1}-\eqref{lemma51-6} that
\begin{align*}
\m{E}\sup_{t\in[0,T\wedge \tau^\epsilon_\alpha]}\|\ee\|_{H^1}^2\leq&C(1+\m{E}\int_0^T \|\ee\|_{H^1}^2ds),
\end{align*}
from which and Gronwall inequality implies
\begin{align*}
\m{E}\sup_{t\in[0,T\wedge \tau^\epsilon_\alpha]}\|\ee\|_{H^1}^2\leq&C.
\end{align*}

From \eqref{lemma51-1}, we have
\begin{align}\label{lemma51-7}
\|\ee\|_{H^1}^4\leq \sum_{i=1}^5|J_i|^2.
\end{align}
Then, using H\"older inequality and the estimates of \eqref{lemma51-2}-\eqref{lemma51-6}, we can obtain
\begin{align*}
\m{E}\sup_{t\in[0,T\wedge \tau^\epsilon_\alpha]}\|\ee\|_{H^1}^4\leq&C(1+\m{E}\int_0^{T\wedge \tau^\epsilon_\alpha} \|\ee\|_{H^1}^4ds),
\end{align*}
from which and Gronwall inequality implies
\begin{align*}
\m{E}\sup_{t\in[0,T\wedge \tau^\epsilon_\alpha]}\|\ee\|_{H^1}^4\leq&C.
\end{align*}
The proof is complete.
\end{proof}

\end{lemma}

\begin{lemma}\label{lemma52}
Let the adapted processes $ y^\epsilon,  \mu^\epsilon, a^\epsilon, b^\epsilon,  y, \mu, a$ and $b$ be given in \eqref{xe1}-\eqref{ce1} and \eqref{eta1}-\eqref{eta1x}.
Then, for any positive $T$ and $t\leq \tau^\epsilon_\alpha$
\begin{align}
  & | y^\epsilon(t)-  y(t)|+  |\mu^\epsilon(t)-  \mu(t)|+  |a^\epsilon(t)-a(t)|+   |b^\epsilon(t)-b(t)|\nonumber\\
  \leq& C\|\ee(t)-\eta(t)\|_{L^2}+C\epsilon (1+\sup_{t\in[0,T\wedge \tau^\epsilon_\alpha]}\|\ee(t)\|_{H^1}^4)\label{ec6}.
\end{align}

\begin{proof}
Denote
\begin{align}\label{ec2}
A_0^{-1}=
{\left( \begin{array}{ccc}
 (\p_x\vp_{c_0}, (1-\p_x^2)\p_x\vp_{c_0})^{-1} &0\\
    0 & (\p_c\vp_{c_0}, (1-\p_x^2)\vp_{c_0})^{-1}
  \end{array}
\right )}
\end{align}
\begin{align}\label{ec3}
D=
{\left( \begin{array}{ccc}
   -(\sigma\p_x\vp_{c_0}, (1-\p_x^2)\p_x\vp_{c_0}))\\
   -(\sigma\p_x\vp_{c_0}, (1-\p_x^2)\vp_{c_0})
  \end{array}
\right )},
\end{align}
and
\begin{align}\label{ec4}
  E=
{\left( \begin{array}{ccc}
 (-\frac12\p_x\mathcal{L}_{c_0}\eta, \p_x\vp_{c_0})\\
    (-\frac12\p_x\mathcal{L}_{c_0}\eta, \vp_{c_0})
\end{array}
\right )}.
\end{align}
Then
\begin{align}\label{ec1}
  A_0^{-1}E=
{\left( \begin{array}{ccc}
    y (t) \\
   a (t)
\end{array}
\right )}:=Y (t)
  \;\;\hbox{and}\;\;
A_0^{-1}D=
{\left( \begin{array}{ccc}
   \mu(t) \\
  b(t)
  \end{array}
\right )}.
\end{align}

It follows from \eqref{me8} and \eqref{ec1} that
\begin{align}\label{ec5}
  Y^\epsilon-Y =&
{\left( \begin{array}{ccc}
   y^\epsilon- y \nonumber\\
   a^\epsilon-a
\end{array}
\right )}=(\tilde{A}^\epsilon)^{-1}E^\epsilon- A_0^{-1}E\nonumber\\
=&((\tilde{A}^\epsilon)^{-1}-A_0^{-1})E^\epsilon+A_0^{-1}(E^\epsilon-E)\nonumber\\
=&(\tilde{A}^\epsilon)^{-1}(A_0-\tilde{A}^\epsilon)A_0^{-1}E^\epsilon+A_0^{-1}(E^\epsilon-E).
\end{align}
Let $v^\epsilon=\ee-\eta.$ Then, using \eqref{ec10}-\eqref{ecg}, we get
\begin{align}\label{ec4}
  E^\epsilon-E=&
{\left( \begin{array}{ccc}
 (-\frac12\p_x\mathcal{L}_{c^\epsilon}\eta^\epsilon+\frac12\p_x\mathcal{L}_{c_0}\eta, \p_x\vp_{c_0})-\epsilon (f(\ee), (1-\p_x^2)\p_x\vp_{c_0})\\
    (-\frac12\p_x\mathcal{L}_{c^\epsilon}\eta^\epsilon+\frac12\p_x\mathcal{L}_{c_0}\eta, \vp_{c_0})-\epsilon (f(\ee), (1-\p_x^2)\vp_{c_0})
\end{array}
\right )}\nonumber\\
= &
{\left( \begin{array}{ccc}
 (-\frac12\p_x\mathcal{L}_{c_0}v^\eps-\frac12\p_xg(\ee), \p_x\vp_{c_0})-\epsilon (f(\ee), (1-\p_x^2)\p_x\vp_{c_0})\\
    (-\frac12\p_x\mathcal{L}_{c_0}v^\eps-\frac12\p_xg(\ee), \vp_{c_0})-\epsilon (f(\ee), (1-\p_x^2)\vp_{c_0})
\end{array}
\right )}
\end{align}
Similar to \eqref{ec13-3}, \eqref{ec13-4} and \eqref{ec14}, we obtain
\begin{align}\label{ec7}
|E^\epsilon-E|\leq C\epsilon(1+\sup_{t\in[0,T\wedge\tau_\alpha^\epsilon]} \|\ee\|_{H^1}^4)+C \|v^\epsilon\|_{L^2}^2.
\end{align}
Since $A_0-\tilde{A}^\epsilon=O(|c^\epsilon-c_0|+\|\epsilon\ee\|_{H^1})=\epsilon O(\|\ee\|_{H^1})$, it follows from \eqref{ec5}-\eqref{ec7} that
\begin{align}\label{ec5z}
  |y^\epsilon- y|+  | a^\epsilon-a|\leq C\|v^\epsilon\|_{L^2}^2+C\epsilon (1+\sup_{t\in[0,T\wedge\tau_\alpha^\epsilon]}\|\ee\|_{H^1}^4).
\end{align}

Now, we estimate \eqref{ec6a}. Similar to \eqref{ec5},
\begin{align*}
  {\left( \begin{array}{ccc}
    \mu^\epsilon-\mu \nonumber\\
   b^\epsilon-b
\end{array}
\right )}
=&(\tilde{A}^\epsilon)^{-1}(A_0-\tilde{A}^\epsilon)A_0^{-1}D^\epsilon+A_0^{-1}(D^\epsilon-D).
\end{align*}
Since
\begin{align}
D^\epsilon-D=
{\left( \begin{array}{ccc}
   -(\sigma\p_x(\vp_{c^\epsilon}-\vp_{c_0})+\eps\sigma\ee_x, (1-\p_x^2)\p_x\vp_{c_0})\\
   -(\sigma\p_x(\vp_{c^\epsilon}-\vp_{c_0})+\eps\sigma\ee_x, (1-\p_x^2)\vp_{c_0})
  \end{array}
\right )},
\end{align}
we have, from \eqref{me11a} and H\"older inequality,
 \begin{align}
|D^\epsilon-D| \leq C\epsilon  \|\ee\|_{H^1}.
\end{align}
Thus, we can get \eqref{ec6a} as that of \eqref{ec6}. The proof is complete.
\end{proof}

\end{lemma}

\begin{proof}[Proof of Theorem \ref{th3}]

By \eqref{i5}, we have
\begin{align}\label{ec10}
\p_x\mathcal{L}_{c^\epsilon}\eta^\epsilon=&\p_x\mathcal{L}_{c_0}\eta^\epsilon  +g(\eta^\epsilon),
\end{align}
where
\begin{align}\label{ecg}
g(\eta^\epsilon)=& -2\p_x[(c^\epsilon-c_0-\vp_{c^\epsilon}+\vp_{c_0})\p_x]\ee -6\p_x((\vp_{c^\epsilon}-\vp_{c_0})\ee)\nonumber\\& +2\p_x(\p_x^2(\vp_{c^\epsilon}-\vp_{c_0})\ee).
\end{align}
Let $v^\epsilon=\ee-\eta.$ Then by \eqref{me1}, \eqref{eta} and \eqref{ec10}
\begin{align}\label{ec9}
dv^\epsilon=&\frac12(1-\p_x^2)^{-1}\p_x\mathcal{L}_{c_0}v^\epsilon dt +\frac12(1-\p_x^2)^{-1}g(\ee)dt
+( y^\epsilon \p_x\vp_{c^\epsilon}- y\p_x\vp_{c_0}) dt\nonumber\\
&-(a^\epsilon\p_c\vp_{c^\epsilon}-a\p_c\vp_{c_0}) dt
+\epsilon y^\epsilon \ee_x dt+\epsilon f(\ee)dt+ h^\eps\diamond dL(t),
\end{align}
where
\begin{align*}
h^\epsilon=\sigma\p_x(\vp_{c^\epsilon}-\vp_{c_0})+\p_x\vp_{c^\epsilon} \mu^\epsilon-\p_x\vp_{c_0} \mu-\p_c\vp_{c^\epsilon}b^\epsilon+\p_c\vp_{c_0}b
+\epsilon (\sigma\ee_x+\ee_x\mu^\epsilon).
\end{align*}
By H\"older inequality,
\begin{align}\label{ec9h1}
\|h^\epsilon\|_{L^2}\leq C(|c^\eps-c_0|+|b^\eps-b|+|\mu^\eps-\mu|+\eps\|\ee\|_{H^1}):=\gamma^\eps.
\end{align}

Let $\beta(t,x)=\Phi_2(t,z,v)$ solves the following differential equation
\[
\frac{d\beta}{dt}=zh^\eps,\;\;\beta(0,x)=v(x).
\]
Then $\Phi_2(t,z,v)=v+z\int_0^th^\eps(r)dr$. Using H\"older inequality and \eqref{ec9h1}, we have
\begin{align*}
&\|\beta(t,x)\|_{L^2}^2=\|\beta(0,x)\|_{L^2}^2+2z\int_0^t(\beta, h^\eps)ds\\
\leq&\|v(x)\|_{L^2}^2+C|z|\int_0^t\|\beta\|_{L^2}\|h^\eps\|_{L^2}ds\\
\leq&\|v(x)\|_{L^2}^2+C|z|\int_0^t\|\beta\|_{L^2}^2+\|h^\eps\|^2_{L^2}ds\\
\leq&\|v(x)\|_{L^2}^2+C|z|\gamma^{\eps2}+C|z|\int_0^t\|\beta\|_{L^2}^2ds.
\end{align*}
The Gronwall's inequality implies
\begin{align}\label{ec9h2}
\|\beta(t,x)\|_{L^2}^2=\|\Phi_2(t,z,v)\|_{L^2}^2\leq(\|v(x)\|_{L^2}^2+C|z|\gamma^{\eps2})e^{C|z|t}.
\end{align}

Applying It\^o formula Lemma \ref{ito2}, we have
\begin{align}\label{ec12}
&\|v^\epsilon\|_{L^2}^2
=\int_0^t(\frac12(1-\p_x^2)^{-1}\p_x\mathcal{L}_{c_0}v^\epsilon, v^\epsilon)ds
+\int_0^t(\frac12(1-\p_x^2)^{-1}g(\ee), v^\epsilon)ds\nonumber\\
&+\int_0^t(( y^\epsilon \p_x\vp_{c^\epsilon}- y\p_x\vp_{c_0}), v^\epsilon)ds
-\int_0^t((a^\epsilon\p_c\vp_{c^\epsilon}-a\p_c\vp_{c_0}), v^\epsilon)ds\nonumber\\
&-\int_0^t(\epsilon y^\epsilon \p_x\ee +\epsilon f(\ee), v^\epsilon)ds\
+\int_0^t\int_{\m{Z}}
[\|\Phi_2(1,z,v^\epsilon(s-))\|_{L^2}^2-\|v^\epsilon(s-)\|_{L^2}^2]\tilde{\mathcal{N}}(ds,dz)\nonumber\\
&+\int_0^t\int_{\m{Z}}
[\|\Phi_2(1,z,v^\epsilon(s))\|_{L^2}^2-\|v^\epsilon\|_{L^2}^2-2z(v^\epsilon, h^\eps)]\vartheta(dz)ds\nonumber\\
=:&\sum_{i=1}^7K_i.
\end{align}
Using \eqref{i5}, we can write $K_1=K_{11}+K_{12}$, where
\begin{align*}
K_{11}=&\int_0^t(-(1-\p_x^2)^{-1}\partial_x^2((c_0-\vp_{c_0})\p_xv^\epsilon), v^\epsilon)ds,\\
K_{12}=&\int_0^t(-3(1-\p_x^2)^{-1}\p_x(\vp_{c_0}v^\epsilon)
+(1-\p_x^2)^{-1}\p_x(\p_x^2\vp_{c_0}v^\epsilon)\nonumber\\
&+(c_0-2k)(1-\p_x^2)^{-1}\p_xv^\epsilon, v^\epsilon)ds.
\end{align*}
Since $-(1-\p_x^2)^{-1}\partial_x^2=I-(1-\p_x^2)^{-1}$, using integration by parts, Cauchy inequality and embedding theorem, we have
\begin{align}\label{ec13-3}
K_{11}=&\int_0^t((c_0-\vp_{c_0})\p_xv^\epsilon, v^\epsilon)ds-\int_0^t((1-\p_x^2)^{-1}((c_0-\vp_{c_0})\p_xv^\epsilon), v^\epsilon)ds\nonumber\\
=&-\frac12\int_0^t((c_0-\p_x\vp_{c_0})v^\epsilon, v^\epsilon)ds-\int_0^t((1-\p_x^2)^{-1}\p_x((c_0-\vp_{c_0})v^\epsilon), v^\epsilon)ds\nonumber\\
&+\int_0^t((1-\p_x^2)^{-1}((c_0-\p_x\vp_{c_0})v^\epsilon), v^\epsilon)ds\nonumber\\
\leq & C(1+\|\vp_{c_0}\|_{H^2})\int_0^t\|v^\epsilon\|_{L^2}^2ds\nonumber\\
\leq&C\int_0^t\|v^\epsilon\|_{L^2}^2ds,
\end{align}
and
\begin{align}\label{ec13-4}
K_{12}\leq&C\int_0^t\|v^\epsilon\|_{L^2}^2ds.
\end{align}
By \eqref{ecg}, integration by parts, Corollary \ref{co41}, H\"older and Young inequalities,
\begin{align}\label{ec14}
K_2\leq&\int_0^t \|(c^\epsilon-c_0-\vp_{c^\epsilon}+\vp_{c_0})\|_{L^\infty}\|\p_x\ee\|_{L^2}\|v^\epsilon\|_{L^2}ds\nonumber\\
&
+\int_0^t \|\vp_{c^\epsilon}-\vp_{c_0}\|_{L^\infty}\|\ee\|_{L^2}\|v^\epsilon\|_{L^2}ds\nonumber\\
&+\int_0^t \|\p_x^2\vp_{c^\epsilon}-\p_x^2\vp_{c_0}\|_{L^\infty}\|\ee\|_{L^2}\|v^\epsilon\|_{L^2}ds\nonumber\\
\leq&C\int_0^t |c^\epsilon-c_0|^2\|\ee\|_{H^1}^2+\|v^\epsilon\|_{L^2}^2ds\nonumber\\
\leq&C\epsilon(1+\sup_{t\in[0,T]} \|\ee\|_{H^1}^6)+C\int_0^t \|v^\epsilon\|_{L^2}^2ds.
\end{align}
It follows from \eqref{ec6} that
\begin{align}\label{ec15}
K_3+K_4\leq C\int_0^t\|v^\epsilon\|_{L^2}^2ds+C\epsilon (1+\sup_{t\in[0,T]}\|\ee\|_{H^1}^4).
\end{align}
By Lemma \ref{lm43}, H\"older and Young inequalities,
\begin{align}\label{ec16}
K_5\leq& C\epsilon^2 \int_0^t(| y^\epsilon|^2\|\p_x\ee\|_{L^2}^2+\|f(\ee)\|_{L^2}^2)ds+C\int_0^t\|v^\epsilon\|_{L^2}^2ds\nonumber\\
\leq&C\epsilon^2\sup_{t\in[0,T\wedge \tau^\epsilon_\alpha]}  \|\ee\|_{H^1}^4+C\int_0^t\|v^\epsilon\|_{L^2}^2ds.
\end{align}
By BDG inequality, \eqref{ec6a} and \eqref{ec9h2}, we have
\begin{align}\label{ec17}
&\m{E}\sup_{t\in[0,T\wedge \tau^\epsilon_\alpha]}K_6\nonumber\\
\leq &C\m{E}\int_0^{T\wedge\tau_\alpha^\epsilon}\left(\int_{\m{Z}}
[(\|v^\eps\|_{L^2}^2+C|z|\gamma^{\eps2})e^{C|z|t}+\|v^\epsilon(s)\|_{L^2}^2]^2\vartheta(dz)\right)^{1/2}ds\nonumber\\
\leq &C\me\gamma^{\eps2}
+C\m{E}\int_0^{T\wedge\tau_\alpha^\epsilon}\|v^\epsilon(s)\|_{L^2}^2ds.
\end{align}
Similarly, we have
\begin{align}\label{ec17}
&\m{E}\sup_{t\in[0,T\wedge \tau^\epsilon_\alpha]}K_7
\leq C\me\gamma^{\eps2}
+C\m{E}\int_0^{T\wedge\tau_\alpha^\epsilon}\|v^\epsilon(s)\|_{L^2}^2ds.
\end{align}
It follows from \eqref{ec12}-\eqref{ec17} that
\begin{align*}
\m{E}\sup_{t\in[0,T\wedge \tau^\epsilon_\alpha]}\|v^\epsilon\|_{L^2}^2\leq&C\m{E}\int_0^{T\wedge\tau_\alpha^\epsilon}\|v^\epsilon\|_{L^2}^2ds+C\epsilon(1+\m{E}\sup_{t\in[0,T\wedge \tau^\epsilon_\alpha]}\|\ee\|_{H^1}^4)+C\me\gamma^{\eps2},
\end{align*}
from which and Gronwall inequality implies
\begin{align}\label{ec18}
\m{E}\sup_{t\in[0,T\wedge \tau^\epsilon_\alpha]} \|v^\epsilon\|_{L^2}^2\leq&C\epsilon(1+\m{E}\sup_{t\in[0,T\wedge \tau^\epsilon_\alpha]}\|\ee\|_{H^1}^4)e^{CT}+C\me\gamma^{\eps2}.
\end{align}
Then, by Lemma \ref{lemma51}, it follows that $v^\epsilon\to0$ as $\epsilon\to0$ in probability in the space $\m{D}([0,T\wedge \tau^\epsilon_\alpha]; L^2)$.
\end{proof}

\smallskip

\appendix

%
%
%
%
%
%
%

\smallskip

\noindent {\bf Acknowledgements.}
We are indebted to the referee for careful reading of the manuscript
and for their comments, which have improved the present work.
This work is partially supported by
China NSF Grant Nos. 12171084, 12071434, 12071435, Zhejiang Provincial NSF of China under Grant No. LZJWY22E060002 and the fundamental Research Funds for the Central Universities No. 2242022R10013.





\end{document}